\documentclass[a4paper, reqno, twoside]{amsart}
\usepackage{etex}
\usepackage{amsfonts,amssymb,graphics,amsthm,amsmath}
\usepackage{latexsym}
\usepackage[2emode]{psfrag}
\usepackage{mathrsfs}
\usepackage[all]{xy}
\usepackage{enumerate}
\usepackage[latin1]{inputenc}
\usepackage{lscape}
\usepackage{a4wide}
\usepackage{txfonts}
\usepackage{multicol}
\usepackage[bookmarks=false]{hyperref}
\usepackage{wasysym}

\usepackage[all]{pstricks}
\usepackage{pst-text,pst-node,pst-coil}
\usepackage{tikz}

\usepackage{cancel}
\usepackage[normalem]{ulem}
\usetikzlibrary{arrows,shapes,backgrounds}

\usepgflibrary{arrows}
\usetikzlibrary{topaths}
\usetikzlibrary{er}

\newtheorem{theorem}{\sc Theorem}[section]
\newtheorem{proposition}[theorem]{\sc Proposition}

\newtheorem{lemma}[theorem]{\sc Lemma}
\newtheorem{corollary}[theorem]{\sc Corollary}
\theoremstyle{definition}
\newtheorem{definition}[theorem]{\sc Definition}

\newtheorem{example}[theorem]{\sc Example}

\theoremstyle{remark}
\newtheorem{remark}[theorem]{\sc Remark}

\newcommand{\tensor}[1]{\otimes_{\scriptscriptstyle{#1}}}

\newcommand{\fk}[1]{\mathfrak{#1}}

\renewcommand{\hom}[3]{\mathrm{Hom}_{#1}\left(#2,\,#3\right)}

\newcommand{\circop}{\circ ^{\Sscript{\mathrm{op}}}}

 \newcommand{\id}[1]{\mathrm{id}_{\Sscript{#1}}}

\newcommand{\Sscript}[1]{\scriptscriptstyle{#1}}
\newcommand{\algk}{\mathsf{Alg}_{\Sscript{\Bbbk}}}
\newcommand{\coalgk}{\mathsf{Coalg}_{\Sscript{\Bbbk}}}
\newcommand{\bialgk}{\mathsf{Bialg}_{\Sscript{\Bbbk}}}
\newcommand{\liek}{\mathsf{Lie}_{\Sscript{\Bbbk}}}
\newcommand{\coliek}{\mathsf{LieCo}_{\Sscript{\Bbbk}}}
\newcommand{\nalgk}{\mathsf{NAlg}_{\Sscript{\Bbbk}}}
\newcommand{\ncoalgk}{\mathsf{NCoalg}_{\Sscript{\Bbbk}}}
\newcommand{\qbialgk}{\mathsf{QBialg}_{\Sscript{\Bbbk}}}
\newcommand{\dqbialgk}{\mathsf{DQBialg}_{\Sscript{\Bbbk}}}
\newcommand{\sdqbialgk}{\mathsf{SDQBialg}_{\Sscript{\Bbbk}}}
\newcommand{\Nalgk}{\mathsf{NAlg}\big(\mathsf{Coalg}_{\Sscript{\Bbbk}}\big)}
\newcommand{\Ncoalgk}{\mathsf{NCoalg}\big(\mathsf{Alg}_{\Sscript{\Bbbk}}\big)}
\newcommand{\vectk}{\mathsf{Vect}_{\Sscript{\Bbbk}}}

\newcommand{\cP}{{\mathcal P}}

\newcommand{\cU}{{\mathcal U}}

\newcommand{\varfun}[3]{#1 \colon #2 \to #3}
\newcommand{\lfun}[5]{{#1} \colon {#2} \longrightarrow {#3}, \,  {#4} \longmapsto {#5}}

\newcommand{\K}{\Bbbk}

\begin{document}
\allowdisplaybreaks

\title[Functorial Constructions for Non-associative Algebras]{Functorial Constructions for Non-associative Algebras with Applications to Quasi-bialgebras}
\author{Alessandro Ardizzoni}
\address{University of Turin, Department of Mathematics ``Giuseppe Peano'', via Carlo Alberto 10, I-10123 Torino, Italy}  \email{alessandro.ardizzoni@unito.it}
\urladdr{sites.google.com/site/aleardizzonihome}
\author{Laiachi El Kaoutit}
\address{Universidad de Granada, Departamento de \'{A}lgebra. Facultad de Educaci\'{o}n, Econon\'ia y Tecnolog\'ia de Ceuta. Cortadura del Valle, s/n. E-51001 Ceuta, Spain}
\email{kaoutit@ugr.es}
\urladdr{http://www.ugr.es/~kaoutit}
\author{Paolo Saracco}
\address{University of Turin, Department of Mathematics ``Giuseppe Peano'', via Carlo Alberto 10, I-10123 Torino, Italy}
\email{p.saracco@unito.it}
\urladdr{sites.google.com/site/paolosaracco}
\date{\today}
\subjclass[2010]{Primary 17A01, 17B62; Secondary 16T99, 18D25.}
\thanks{This paper was written while A. Ardizzoni and P. Saracco were members of the ``National Group for Algebraic and Geometric Structures, and their Applications'' (GNSAGA - INdAM). The research of L. El Kaoutit is supported by the Spanish Ministerio de Econom\'ia y Competitividad and the European Union, grant MTM2013-41992-P.}

\begin{abstract}
The aim of this paper is to establish a contravariant adjunction between the category of quasi-bialgebras and a suitable full subcategory of dual quasi-bialgebras, adapting the notion of finite dual to this framework.  Various functorial constructions involving non-associative algebras and non-coassociative coalgebras are then carried out.  Several examples  illustrating our methods are expounded as well.
\end{abstract}

\keywords{Non-(co)associative (co)algebras;  Quasi-bialgebras; Contravariant adjuntions; Finite duals.}
\maketitle

%\begin{small}
%\tableofcontents
%\end{small}

\pagestyle{headings}

\vspace{-0,5cm}

\section{Introduction}

%{\green This is an attempt for the Introduction. Please make any change you estimate pertinent.}
Algebras and coalgebras are dual notions, in the sense that the latter ones are obtained from the first ones by reversing the structure arrows, that is using the opposite base category. Furthermore, over vector spaces there is a contravariant adjunction (or duality) between the category of algebras and that of coalgebras, whose functors are described as follows. In one direction, to each coalgebra one associates, in a functorial way, its convolution algebra. In the other direction, to each algebra one associates, in a similar way, its topological dual (i.e.~ finite dual) coalgebra.  This adjunction descends in fact to the category of bialgebras (and in particular to Hopf algebras), and also establishes a contravariant adjunction between the category of Lie algebras and the category of Lie coalgebras. All these adjunctions and other ones are captured by the following diagram
\begin{equation}\label{diagramma1}
\xymatrix@R=25pt{ \algk  \ar@{->}^-{(-)^{\circ}}[rrr] &  & & \ar@<+0.9ex>@{->}^-{(-)^*}[lll]  \coalgk  \\    \bialgk  \ar@{->}[u] \ar@{->}^-{(-)^{\circ}}[rrr]   \ar@<+0.9ex>@{->}^-{\cP}[d] & & & \ar@{->}_-{\cP^{\Sscript{c}}}[d] \bialgk \ar@{->}[u] \ar@<+0.9ex>@{->}^-{(-)^{\circ}}[lll] \\ \liek \ar@{->}^-{\cU}[u]  \ar@{->}^-{(-)^{\bullet}}[rrr]  & & & \ar@<+0.9ex>@{->}^-{(-)^{*}}[lll] \coliek \ar@<-0.9ex>@{->}_-{\cU^{\Sscript{c}}}[u] }
\end{equation}
where $\Bbbk$ denotes the base field, and the notations for the involved categories as well as the ones used in the sequel,  are summarized in Table \ref{tavola1} below.

%All vector spaces are over $\Bbbk$. By a \emph{non-(co)associative (co)algebra}
%we mean a (co)unital but \emph{not necessarily (co)associative (co)algebra} over $\Bbbk$.

\begin{table}[ht]\label{tavola1}
\caption{Notations for the handled categories.}
\centering
\begin{tabular}{l l  l}
\hline \hline
Abbreviation  & & The category  of \\  [0.5ex]
\hline
$\vectk$ & $\cdots\cdots$ & vector spaces  \\
$\algk$ & $\cdots\cdots$ & associative algebras \\
$\coalgk$ & $\cdots\cdots$  &  coassociative coalgebras \\
$\bialgk$  & $\cdots\cdots$   &  bialgebras \\
$\liek$ & $\cdots\cdots$  &  Lie algebras \\
$\coliek$ & $\cdots\cdots$   &  Lie coalgebras \\
$\nalgk$ & $\cdots\cdots$   & non-associative algebras\\
$\ncoalgk$ & $\cdots\cdots$   & non-coassociative coalgebras \\
$\Nalgk$ & $\cdots\cdots$ & non-associative algebras inside the monoidal category of coalgebras  \\
$\Ncoalgk$ & $\cdots\cdots$ & non-coassociative coalgebras inside the monoidal category of algebras \\
$\qbialgk$ & $\cdots\cdots$   & quasi-bialgebras \\
$\dqbialgk$ & $\cdots\cdots$  & dual quasi-bialgebras  \\  $\sdqbialgk$ & $\cdots\cdots$  &  split dual quasi-bialgebras \\ & & (i.e.~the 3-cocycle splits as a finite sum of tensor products of linear maps)  \\
[0.5ex] \hline\hline
\end{tabular} \label{table:1}
\end{table}

It is noteworthy to mention that in diagram \eqref{diagramma1}, the functor  $\cP: \bialgk \to \liek$ associates to each bialgebra its Lie algebra of primitive elements, and its left adjoint $\cU: \liek \to \bialgk$ is the universal enveloping algebra functor. Note that in characteristic zero there is a natural isomorphism $\cP \cU \cong \id{\liek}$, see \cite[Theorem 5.18]{MilnorMoore}.
The functor $\cP^{\Sscript{c}}$ is the one given by the vector space of indecomposables, see e.g.~ \cite[Definition 1.9]{Michaelis-primitive} where this functor is denoted by $Q$, and  $\cU^{\Sscript{c}}$ is a functor defined by Michaelis. The adjunction $\cP^{\Sscript{c}} \dashv \cU^{\Sscript{c}}$ is established in \cite[Theorem 3.11]{Michaelis-primitive} for Hopf algebras instead of bialgebras (the same proof can be adapted to our case, as the antipode is not used therein). The bottom contravariant adjunction is established in \cite[Theorem 3.7]{Michaelis-primitive}.  For the top horizontal adjunction see e.g. \cite[Theorem 6.0.5]{Sweedler}. Concerning the middle horizontal adjunction, the finite dual yields an endofunctor of $\bialgk$ in view of \cite[Section 6.2]{Sweedler}. Moreover, this comes out to be adjoint to itself as in case of Hopf algebras (cf. e.g. \cite[page 87]{Abe}).

Quasi-bialgebras are generalization of ordinary bialgebras,  in which the constraint of coassociativity at the coalgebra level is weakened. Dual quasi-bialgebras are in a certain sense a dual notion, which can also be seen as a generalization of bialgebras, by affecting this time the associativity constraint.

The main aim of this paper is  to  investigate the second horizontal adjunction of diagram \eqref{diagramma1} in the context of quasi and dual quasi-bialgebras. Explicitly, we establish a contravariant adjunction between the category of quasi-bialgebras and  the one of split dual quasi-bialgebras (a certain full subcategory of the category of dual quasi-bialgebras) here introduced.  To do so, we investigate how  some of the adjunctions represented in diagram \eqref{diagramma1} extend to the wider framework of non-(co)associative (co)algebras, as in diagram \eqref{diagramma2}. We just point out here that the upper adjunction in diagram \eqref{diagramma2} already appeared in \cite[page 4700]{Anquela}.

%\subsection{Description of the main results.}\label{ssec:02}

\begin{equation}\label{diagramma2}
\xymatrix@R=30pt{  & \nalgk \ar@<+0.9ex>@{->}^-{(-)^{\bullet}}[rrr]      &  & & \ncoalgk \ar@{->}^-{(-)^*}[lll] \\  \liek \ar@{->}[ur] \ar@<+0.75ex>@{->}^-{(-)^{\bullet}}[rrr]  & & & \ar@{->}^-{(-)^*}[lll] \coliek \ar@{->}[ur] & \\&  \Nalgk \ar@{->}[uu]|(.52){\hole}  \ar@<+0.9ex>@{->}^-{(-)^{\bullet}}[rrr] &  & &  \ar@{->}^-{(-)^{\circ}}[lll] \Ncoalgk \ar@{->}[uu] \\ \dqbialgk \ar@{->}[ur] & & &  \ar@{->}^-{(-)^{\circ}}[lll]  \qbialgk  \ar@{->}[ur] &  \\ & \sdqbialgk  \ar@{_{(}->}[ul] \ar@{->}^-{(-)^{\bullet}}[rrr] &  &  & \ar@<+0.9ex>@{->}^-{(-)^{\circ}}[lll]  \qbialgk \ar@{=}[ul] }
\end{equation}

By a \emph{non-associative algebra}
we mean a unital but \emph{not necessarily associative algebra over $\Bbbk$}, i.e.~a vector space $A$ endowed with two linear maps $m:A\otimes A\longrightarrow A$, $a\otimes b\longmapsto ab$ (the multiplication)  and $u:\Bbbk \longrightarrow A$, $k\longmapsto k1_{A}$ (the unit) such that $a1_{A}=a=1_{A}a$, for every $a\in A$. A similar terminology is used for coalgebras.

% \paolo{[IDEA: Since we know that the primitives of a dual quasi-bialgebra still define an ordinary Lie algebra, maybe we can check if the same is true for the indecomposable and add the two arrows to this diagram.]}
% \paolo{???}

\section{The construction of the finite dual of a non-associative algebra and examples.}\label{sec:nass}
In this section we give the main construction of the paper. More explicitly, starting from a non-associative algebra, we are able to construct a non-coassociative coalgebra, which in the associative case coincides with the so called \emph{finite dual coalgebra}, see  \cite{Abe, Montgomery, Sweedler} and \cite{Abuhlail/Gomez-Torrecillas/Wisbauer:2000} for coalgebras over commutative rings. This is the largest coalgebra inside the linear dual of the underlying vector space of the initial algebra. To illustrate our techniques,  we  include two basic examples concerning alternative and (special) Jordan algebras.

Given a vector space $V$, we denote by $V^*\coloneqq  \hom{\Bbbk}{V}{\Bbbk}$ its linear dual. The unadorned tensor product $\otimes$ stands for $\tensor{\Bbbk}$. The identity morphism of a vector space $V$ will be denoted by $\id{V}$ or by $V$ itself.

\subsection{Good subspace of linear dual vector space.}\label{ssec:goods}
Given two vector spaces $V$ and $W,$ we can consider the canonical natural  injection
\begin{equation}\label{Eq:varphi}
\varphi _{V,W}:V^{\ast }\otimes W^{\ast }\longrightarrow \left( V\otimes
W\right) ^{\ast }, \quad \Big( f\otimes g\longmapsto \big[ v\otimes w \mapsto f(v)g(w)\big] \Big),
\end{equation}
which is clearly a natural isomorphism over finite-dimensional vector spaces.

Let $\left( A,m,u\right) $ be a non-associative algebra, see
e.g. \cite[page 428]{Bourbaki-Alg1}. Mimicking \cite[page 13]{Michaelis-LieCoalg}, a subspace $V\subseteq A^{\ast }$ is called \emph{good} in case $m^{\ast }\left( V\right) \subseteq \varphi _{A,\,A}\left( V\otimes V\right)$, where $m^*: A^* \to (A\tensor{}A)^*$ is the dual of the multiplication map $m$.  For instance,
let $I$ be an ideal of $A$, that is  a vector subspace of $A$ stable under both left and right $A$-actions:  for every $a \in A$, we have $aI \subseteq I$, $Ia \subseteq I$, see  \cite[page 430]{Bourbaki-Alg1}. Assume that $A/I$ is finite dimensional as a vector space. Set $V=(A/I)^*$ which we identify with a subspace of $A^*$. One can show that $V$ is a good subspace of $A^*$.

Let $\mathcal{G}$ denote the set of all good subspaces of $A^{\ast }$ and set
\begin{equation}\label{MicFinDual}
A^{\bullet }\coloneqq \sum_{V\in \mathcal{G}}V.
\end{equation}
By the same proof of \cite[Proposition, page 13]{Michaelis-LieCoalg}, one gets that $A^{\bullet }$ is a good subspace of $A^{\ast }$ and hence it is the maximal good subspace of $A^{\ast }$.

Given two non-associative algebras $A$ and $B$ and a linear map $f: A \to B$ such that $f^*(B^{\bullet}) \subseteq A^{\bullet}$, then we can consider the linear map $f^{\bullet}: B^{\bullet} \to A^{\bullet}$, $h\mapsto f^*(h)$, which is uniquely determined by the commutativity of the following diagram:
\begin{equation}\label{eq:pallino}
\xymatrixrowsep{15pt}
\xymatrix{  B^{\bullet} \ar@{-->}^-{f^{\bullet}}[rr] \ar@{^{(}->}[d]_-{j_B} && A^{\bullet}  \ar@{^{(}->}[d]^-{j_A}  \\ B^* \ar@{->}^-{f^*}[rr] & & A^*  }
\end{equation}
where the vertical arrows are the canonical injections.

If we consider a good subspace $V\subseteq A^{\ast }$,  then we may define a unique map $\Delta _{V}:V\rightarrow V\otimes V$
such that
\begin{equation*}
\varphi _{A,A}\left( \Delta _{V}\left( f\right) \right) =m^{\ast }\left(
f\right) ,\text{ for every }f\in V.
\end{equation*}
In particular, for every $f\in A^{\bullet },a,b\in A,$ $\Delta _{V}\left(
f\right) =\sum f_{1}\otimes f_{2}$ is uniquely determined by
\begin{equation}
f\left( ab\right) =m^{\ast }\left( f\right) \left( a\otimes b\right)
=\varphi _{A,A}\left( \Delta _{V}\left( f\right) \right) \left( a\otimes
b\right) =\sum f_{1}\left( a\right) f_{2}\left( b\right) .
\label{comultiplication}
\end{equation}

\subsection{The coalgebra structure of $A^{\bullet}$ and examples}
Parts of the subsequent lemma find their analogues for associative algebras in \cite[Lemma 6.0.1]{Sweedler} and for Lie algebras in \cite[pages 14-15]{Michaelis-LieCoalg}.

\begin{lemma}\label{lemma:4.11}
For every pair of non-associative algebras $(A,m,u)$ and $%
(B,m^{\prime },u^{\prime })$ and for any morphism $f:A\longrightarrow B$, denote with $f^{\ast }:B^{\ast }\longrightarrow A^{\ast }$ the dual map. Then the dual map $m^{\ast }:A^{\ast }\rightarrow \left( A\otimes
A\right) ^{\ast }$ induces a map $\Delta _{A^{\bullet }}\coloneqq m^{\ast }:A^{\bullet}\rightarrow A^{\bullet }\otimes A^{\bullet }$ and the dual map $u^{\ast}:A^{\ast }\rightarrow \Bbbk ^{\ast }\cong \Bbbk :f\mapsto f\left( 1\right) $ restricts to a map $\varepsilon _{A^{\bullet }}\coloneqq u^{\ast }:A^{\bullet}\rightarrow \Bbbk $ such that $(A^{\bullet },\Delta _{A^{\bullet }},\varepsilon_{A^{\bullet }})$ becomes a non-coassociative coalgebra.
\end{lemma}
\begin{proof}
First observe that $\Delta _{A^{\bullet }}$ exists by definition of the finite dual and satisfies $(\ref{comultiplication})$.
Therefore, let us show that $\varepsilon _{A^{\bullet }}$ is a counit for $\Delta _{A^{\bullet }}$. Pick an element $f\in A^{\bullet }$. For every $a\in A$
we have that:
\begin{eqnarray*}
\left( ((\varepsilon _{A^{\bullet }}\otimes A^{\bullet })\circ \Delta _{A^{\bullet
}})(f)\right) (a) &=&\left( (\varepsilon _{A^{\bullet }}\otimes A^{\bullet
})\left( \sum f_{1}\otimes f_{2}\right) \right) (a)=\left( \sum
f_{1}(1)f_{2}\right) (a) \\
&=&\sum f_{1}(1)f_{2}(a)\overset{(\ref{comultiplication})}{=}f(a), \\
\left( ((A^{\bullet }\otimes \varepsilon _{A^{\bullet }})\circ \Delta _{A^{\bullet
}})(f)\right) (a) &=&\left( (A^{\bullet }\otimes \varepsilon _{A^{\bullet
}})\left( \sum f_{1}\otimes f_{2}\right) \right) (a)=\left( \sum
f_{2}(1)f_{1}\right) (a) \\
&=&\sum f_{2}(1)f_{1}(a)\overset{(\ref{comultiplication})}{=}f(a).
\end{eqnarray*}
whence $(A^{\bullet },\Delta _{A^{\bullet }},\varepsilon _{A^{\bullet }})$ is a
non-coassociative coalgebra in $\vectk$.
\end{proof}

\begin{remark}\label{rem:oclassic}
%\begin{enumerate}
%\item
Let $A$  be an object in  $\nalgk$ and  set
\begin{equation}\label{eq:usualfinitedual}
A^{\circ}=\Big\{g\in A^*\mid \ker(g) \textrm{ contains a finite-codimensional ideal of }A\Big\}.
\end{equation}
Here an ideal $I$ of $A$ is of finite codimension, means that $A/I$ is a finite-dimensional vector space.
For any $f\in A^\circ$, there exists a finite-codimensional ideal $I$ such that $f(I)=0$. Then $f$ belongs to the space $(A/I)^*$, which is identified with a good subspace of $A^*$ as in subsection \ref{ssec:goods}. By equation \eqref{MicFinDual}, this means that $f\in A^\bullet$. We have so proved that $A^\circ\subseteq A^\bullet$.  This fact can also be seen as a consequence of \cite[Theorem (2.6)]{Anquela} which asserts that $A^\circ=\mathrm{Loc}(A^\bullet)$, where the latter denotes the sum of all locally finite subcoalgebras of $A^\bullet$ (recall that a non-coassociative coalgebra $C$ is named \emph{locally finite} if and only if any $x\in C$ lies in some finite-dimensional subcoalgebra $D\subset C$).

%\item

For any $A$ in $\algk$ the finite dual $A^\bullet$ coincides with $A^\circ$. By the foregoing $A^\circ \subseteq A^\bullet$. Conversely, identifying $A^*\otimes A^*$ with $\varphi_{A,\, A}(A^*\otimes A^*)$, if $V\subseteq A^*$ is any good subspace then for every $v\in V$, $m^*(v)\in A^*\otimes A^*$. Therefore, in view of \cite[Proposition 6.0.3]{Sweedler}, $v\in A^\circ$ and hence $V\subseteq A^\circ$. Thus $A^\bullet\subseteq A^\circ$ so that $A^\bullet=A^\circ$.
%\end{enumerate}
\end{remark}

We now provide two examples of finite dual of a non-associative algebra.

\begin{example}\label{exam:alt}\emph{Coalternative coalgebras}. Assume $\Bbbk$ has a characteristic $\neq 2$. Let $A$ be an \emph{alternative  algebra}, that is a not necessarily associative  algebra over $\Bbbk$, which satisfy the following identity
$$ x(yx) \,=\, (xy) x, \quad \text{for every } x, y \in A.$$
Replacing $x$ by $x+z$ one sees that the last equality is equivalent to the identity
$$
x(yz) + z(yx) \,=\, (xy) z + (zy)x, \quad \text{for every } x, y, z \in A.
$$

Denote by $\tau: V\tensor{}W \to W\tensor{}V$ the natural flip map, and set $\tau_{\Sscript{1}}=\tau\tensor{}id$ and $\tau_{\Sscript{2}}=id\tensor{}\tau$.
Consider the finite dual coalgebra $C=A^\bullet$ as in  Lemma \ref{lemma:4.1}.  Then the comultiplication of $C$ satisfies the identity
\begin{equation}\label{Eq:alt}
\Big(id+\big(\tau_1\circ \tau_2\circ \tau_1\big)\Big) \circ \Big( (\Delta\tensor{}C)   -  (C\tensor{}\Delta)   \Big) \circ \Delta  \,\,=\,\,0,
\end{equation}
which over elements, says that for any function $f \in C$, we have
\begin{equation}\label{Eq:alt1}
\sum f_{\Sscript{1,1}}\tensor{} f_{\Sscript{1,2}}\tensor{}f_{\Sscript{2}}   + \sum f_{\Sscript{2}}\tensor{} f_{\Sscript{1,2}}\tensor{}f_{\Sscript{1,1}} \,\,=\,\,   \sum f_{\Sscript{1}}\tensor{} f_{\Sscript{2,1}}\tensor{}f_{\Sscript{2,2}} + \sum f_{\Sscript{2,2}}\tensor{} f_{\Sscript{2,1}}\tensor{}f_{\Sscript{1}}.
\end{equation}

A coalgebra $C$ which satisfies the identity \eqref{Eq:alt} is called a \emph{coalternative coalgebra}.
\end{example}

\begin{example}\emph{Jordan coalgebra}.
Assume $\Bbbk$ has a characteristic $\neq \{2,3\}$. Let $A$ be an \emph{a (special) Jordan  algebra}, that is a not necessarily associative  algebra over $\Bbbk$, which satisfy the following identity
$$ xy\,=\, yx,\quad x^2(yx) \,=\, (x^2y) x, \quad \text{for every } x, y \in A.$$
The second equality above comes out to be equivalent to
$$((xy)z)t + ((xt)z)y + ((ty)z)x = (xy)(zt) + (xt)(zy) + (ty)(zx).$$

Denote by $\tau: V\tensor{}W \to W\tensor{}V$ the natural flip map, and set $\tau_{\Sscript{1}}=\tau\tensor{}id\tensor{}id$, $\tau_{\Sscript{2}}=id\tensor{}\tau\tensor{}id$ and $\tau_{\Sscript{3}}=id\tensor{}id\tensor{}\tau$.
Following \cite[Example (3) page 4709]{Anquela}, we can consider the finite dual coalgebra $C=A^\bullet$ as in  Lemma \ref{lemma:4.1}.  Then the comultiplication of $C$ is cocommutative and satisfies the identity
\begin{equation}\label{Eq:Jordan}
\Big[id+\big(\tau_3\circ \tau_2 \circ \tau_3\big)\,+\, \big(\tau_3\circ \tau_2\circ \tau_1\circ \tau_2\circ \tau_3\big)\Big] \circ \Big[ (\Delta\tensor{}C\tensor{}C) -(C\tensor{}C\tensor{}\Delta)\Big] \circ   (\Delta\tensor{}C) \circ \Delta    \,\,=\,\,0,
\end{equation}
which over elements, says that for any function $f \in C$, we have
\begin{multline}\label{Eq:Jordan1}
\sum f_{\Sscript{1,1,1}}\tensor{} f_{\Sscript{1,1,2}}\tensor{}f_{\Sscript{1,2}}\tensor{}f_{\Sscript{2}}    + \sum f_{\Sscript{1,1,1}}\tensor{} f_{\Sscript{2}}\tensor{}f_{\Sscript{1,2}}\tensor{} f_{\Sscript{1,1,2}}
 + \sum f_{\Sscript{2}}\tensor{} f_{\Sscript{1,1,2}}\tensor{}f_{\Sscript{1,2}}\tensor{} f_{\Sscript{1,1,1}}
\\   \,=\,
\sum f_{\Sscript{1,1}}\tensor{} f_{\Sscript{1,2}}\tensor{}f_{\Sscript{2,1}}\tensor{} f_{\Sscript{2,2}}
+
\sum f_{\Sscript{1,1}}\tensor{} f_{\Sscript{2,2}}\tensor{}f_{\Sscript{2,1}}\tensor{}f_{\Sscript{1,2}} + \sum  f_{\Sscript{2,2}}\tensor{} f_{\Sscript{1,2}}\tensor{}f_{\Sscript{2,1}}\tensor{}f_{\Sscript{1,1}}.
\end{multline}

A cocommutative coalgebra $C$ which satisfies the identity \eqref{Eq:Jordan} is called a \emph{Jordan coalgebra}.
\end{example}

\section{Contravariant adjunction between non-associative algebras and non-coassociative coalgebras.}\label{sec:cadj}
In \cite[page 4700]{Anquela} it is claimed that the contravariant functor $(-)^\bullet$ is the right adjoint of the functor $(-)^*$ from the category $\ncoalgk$ of non-coassociative coalgebras to the category  $\nalgk$ of non-associative algebras (we just point out that their (co)algebras have no (co)unit).
%\sout{The aim of this section is to establish a contravariant adjunction between the category $\ncoalgk$ of non-coassociative coalgebras and the category  $\nalgk$  non-associative algebras.}
Such an adjunction extends the usual contravariant adjunction between algebras and coalgebras, that is the first horizontal adjunction in   diagram \eqref{diagramma1}. For the sake of completeness, and as reference for the sequel, we decided to detail the relevant proofs.
We first check that the construction $A\longmapsto A^{\bullet }$ of Section \ref{sec:nass}  defines a contravariant functor from the category $\nalgk$ to  the category  $\ncoalgk$. Next we show that it is adjoint to the functor $(-)^{\ast }$ defined as in the classical case by using the convolution product. In addition,  we
prove that $(A\otimes B)^{\bullet}\cong A^\bullet\otimes B^\bullet$, for every $A,B$ in $\nalgk$.

\subsection{The functorial construction.}\label{ssec:fc}  Keep the notations of Section \ref{sec:nass}.

\begin{lemma}\label{lemma:4.1}
Let $ f: A \to B$ be a morphism of non-associative algebras and  $f^{\ast }:B^{\ast }\longrightarrow A^{\ast }$ its linear dual map. Then $f^{\ast }(B^{\bullet })\subseteq A^{\bullet }$, whence $f^{\ast }$ induces a map $f^{\bullet }:B^{\bullet }\to A^{\bullet }$, which comes out to be a morphism in $\ncoalgk$. Moreover, the assignments $A\longmapsto A^{\bullet }$ and $f\longmapsto
f^{\bullet }$ establish a functor $$(-)^{\bullet }:\nalgk
\longrightarrow \ncoalgk{}^{\Sscript{\mathrm{op}}}.$$
\end{lemma}
\begin{proof}
Since $f$ is multiplicative, the left-hand side diagram below commutes, so that, by functoriality of $(-)^*$, the right-hand side one commutes, too.
\begin{displaymath}
\xymatrix@C=30pt{
A\otimes A \ar[r]^-{m_A} \ar[d]_-{f\otimes f} & A\ar[d]^-{f} \\
B\otimes B \ar[r]_-{m_B} & B}
\qquad
\xymatrix@C=30pt{
A^* \ar[r]^-{m_A^*} & (A\otimes A)^* \\
B^* \ar[r]_-{m_B^*} \ar[u]^-{f^*} & (B\otimes B)^* \ar[u]_-{(f\otimes f)^*}}
\end{displaymath}
The latter diagram is part of the following bigger one:
\begin{displaymath}
\xymatrix@C=40pt@R=30pt{
 & A^* \ar[rr]^-{m_A^*} & & (A\otimes A)^* \\
B^* \ar[rr]_-{m_B^*} \ar[ur]^-{f^*} & & (B\otimes B)^* \ar[ur]^-{(f\otimes f)^*} & A^*\otimes A^* \ar[u]_-{\varphi_{A,A}} \\
B^\bullet \ar[r]_-{\Delta_{B^\bullet}} \ar @{^(->}[u]^-{j_B} & B^\bullet\otimes B^\bullet \ar @{^(->}[r]_-{j_B\otimes j_B} & B^*\otimes B^* \ar[u]_-{\varphi_{B,B}} \ar[ur]_-{f^*\otimes f^*} & }
\end{displaymath}
which still commutes by definition of $\Delta_{B^\bullet}$ and by naturality of $\varphi_{-,-}$. In particular, for every $g\in B^\bullet$
\begin{equation*}
m^*_A\left(f^*(g)\right)=\varphi_{A,A}\left(\left(f^*\otimes f^*\right)(\Delta_{B^\bullet}(g))\right)\in \varphi_{A,A}\left(f^*\left(B^\bullet\right)\otimes f^*\left(B^\bullet\right)\right)
\end{equation*}
so that $f^*\left(B^\bullet \right)$ is a good subspace of $A^*$. Furthermore, the commutativity of all the other quads in the subsequent diagram implies the commutativity of the one at the bottom, which encodes the comultiplicativity of $f^\bullet$:
\begin{displaymath}
\xymatrix @C=40pt @R=20pt{
B^* \ar[rrr]^-{m_B^*} \ar[dr]_{f^*} & & & (B\otimes B)^* \ar[dl]^{(f\otimes f)^*} \\
 & A^* \ar[r]^-{m_A^*} & (A\otimes A)^* & \\
 & A^\bullet \ar@{^(->}[u]^-{j_A} \ar[r]_-{\Delta_{A^\bullet}} & A^\bullet\otimes A^\bullet  \ar[u]_-{\varphi_{A,A}\circ(j_A\otimes j_A)} & \\
B^\bullet \ar[rrr]_-{\Delta_{B^\bullet}} \ar[ur]^{f^\bullet} \ar @{^(->}[uuu]^-{j_B} & & & B^\bullet\otimes B^\bullet \ar[uuu]_-{\varphi_{B,B}\circ(j_B\otimes j_B)} \ar[ul]_{f^\bullet\otimes f^\bullet} }
\end{displaymath}
Moreover, $f^\bullet$ is counital, since $\varepsilon_{A^\bullet}\circ f^\bullet = (u_A)^\bullet\circ f^\bullet=(f\circ u_A)^\bullet=(u_B)^\bullet=\varepsilon_{B^\bullet}$.
By Lemma \ref{lemma:4.11} it follows that $(-)^\bullet$ actually defines a contravariant functor from $\nalgk$ to $\ncoalgk$.
\end{proof}

\subsection{The contravariant adjunction.}\label{ssec:cadj}
Recall that the assignment $(-)^{\ast }:\coalgk \longrightarrow
\algk$ defines a contravariant functor between the category
of coassociative $\Bbbk $-coalgebras and the category of associative $\Bbbk $-algebras (cf., e.g., \cite[Theorem 6.0.5]{Sweedler}) that easily extends to a contravariant functor $(-)^{\ast }:\ncoalgk \longrightarrow
\nalgk$. In this way, we will show that this functor is in fact an adjoint to the functor $(-)^{\bullet}$ of subsection \ref{ssec:fc}. By restricting these contravariant functors to the full subcategories $\liek$  and $\coliek$, one recovers Michaelis' result \cite[Theorem on page 15]{Michaelis-LieCoalg}.

\begin{proposition}\label{prop:adj1}
Let $(A,m,u)$ be a non-associative algebra and $(C,\Delta ,\varepsilon )$ a non-coassociative coalgebra. We have a natural  isomorphism
\begin{equation}\label{eq:nonassadj}
\Phi_{\Sscript{(A,C)}}:\nalgk(A,C^{\ast})\longrightarrow \ncoalgk(C,A^{\bullet })
\end{equation}

Therefore the functor $(-)^{\bullet }:\nalgk \longrightarrow (\ncoalgk)^{\mathrm{op}}$ is left adjoint to $(-)^{\ast }:(\ncoalgk)^{\mathrm{op}}\longrightarrow \nalgk$.
\end{proposition}
\begin{proof}
Denote by $\chi _{V}:V\longrightarrow V^{\ast \ast }$ the canonical injection, defined for every vector space $V$ by $\chi _{V}(v)(f)=f(v)$ for all $v\in V$ and $f\in V^{\ast }$. Recall that $\chi :(-)\longrightarrow (-)^{\ast \ast }$ is a natural transformation and let us check that if $C$ is a coalgebra then $\chi _{C}\left( C\right) \subseteq C^{\ast \bullet }$ (compare with \cite[Note on page 15]{Michaelis-LieCoalg}). This follows once proved that $\chi _{C}\left( C\right) $ is a good subspace of $C^{\ast \ast }$. Given $c\in C$ and $\phi ,\psi \in C^{\ast }$ we have that:
\begin{eqnarray*}
(m_{C^{\ast }})^{\ast }(\chi _{C}\left( c\right) )(\phi \otimes \psi )
&=&\chi _{C}\left( c\right) (m_{C^{\ast }}(\phi \otimes \psi ))=m_{C^{\ast
}}(\phi \otimes \psi )(c) \\
&=&\sum \phi (c_{1})\psi (c_{2})=\sum \chi _{C}\left( c_{1}\right) (\phi
)\chi _{C}\left( c_{2}\right) (\psi ) \\
&=&\varphi _{C^{\ast },C^{\ast }}\left( \sum \chi _{C}\left( c_{1}\right)
\otimes \chi _{C}\left( c_{2}\right) \right) (\phi \otimes \psi )
\end{eqnarray*}
so that $(m_{C^{\ast }})^{\ast }(\chi _{C}\left( c\right) )=\varphi
_{C^{\ast },C^{\ast }}\left( \sum \chi _{C}\left( c_{1}\right) \otimes \chi
_{C}\left( c_{2}\right) \right) $ for all $c\in C$ and
\begin{equation*}
(m_{C^{\ast }})^{\ast }(\chi _{C}\left( C\right) )\subseteq \varphi
_{C^{\ast },C^{\ast }}(\chi _{C}\left( C\right) \otimes \chi _{C}\left(
C\right) )
\end{equation*}
so that $\chi_C(C)$ is good by definition.  Note also that we have just proved that
\begin{equation}\label{eq:chiC}
(m_{C^{\ast }})^{\ast }\circ \chi_{C}=\varphi _{C^{\ast },C^{\ast }}\circ (\chi _{C}\otimes \chi _{C})\circ
\Delta .
\end{equation}

If we denote by $j_{A}:A^{\bullet }\longrightarrow A^{\ast }$ the inclusion of
the finite dual of a non-associative algebra $A$ into its ordinary
dual, then we have just shown that for any non-coassociative coalgebra $C$, $\chi _{C}$
induces a $\Bbbk $-linear map
$\epsilon _{C}:C\longrightarrow C^{\ast \bullet }$
that is still natural in $C$ and it satisfies
\begin{equation}\label{eq:jC}
j_{C^{\ast }}\circ \epsilon _{C}=\chi _{C}.
\end{equation}
Let us check that $\epsilon _{C}$ actually is a comultiplicative and
counital map. Denote by $\Delta_{C^{\ast \bullet }}$ the comultiplication of $
C^{\ast \bullet }$. This is the only map that satisfies $\varphi _{C^{\ast
},\,C^{\ast }}\circ (j_{C^*}\otimes j_{C^*})\circ \Delta_{C^{\ast \bullet }}=(m_{C^{\ast }})^{\ast }\circ j_{C^{\ast
}}$.

As a consequence:
\begin{eqnarray*}
\varphi _{C^{\ast },\,C^{\ast }}\circ (j_{C^*}\otimes j_{C^*})\circ \Delta_{C^{\ast \bullet }}\circ \epsilon _{C}
&=&(m_{C^{\ast }})^{\ast }\circ j_{C^{\ast }}\circ \epsilon _{C} \stackrel{\eqref{eq:jC}}{=}(m_{C^{\ast
}})^{\ast }\circ \chi _{C} \\
&\stackrel{\eqref{eq:chiC}}{=}&\varphi _{C^{\ast },\,C^{\ast }}\circ (\chi _{C}\otimes \chi _{C})\circ
\Delta \\
&\stackrel{\eqref{eq:jC}}{=}&\varphi _{C^{\ast },\,C^{\ast }}\circ  (j_{C^*}\otimes j_{C^*})\circ  (\epsilon _{C}\otimes \epsilon
_{C})\circ \Delta
\end{eqnarray*}
and, by injectivity of $\varphi _{C^{\ast },\,C^{\ast }}$ and of $j_{C^*}$, we have that $\Delta_{C^{\ast \bullet }}\circ \epsilon _{C}=(\epsilon _{C}\otimes \epsilon _{C})\circ \Delta $.
Moreover, $$\varepsilon_{C^{\ast \bullet}}(\epsilon_C(c))\stackrel{(\dag)}{=}\left(u_{C^\ast}^\bullet(\epsilon_C(c))\right)(1_{\Sscript \Bbbk})=(\epsilon_C(c))\left(u_{C^\ast}(1_{\Sscript \Bbbk})\right)=(\epsilon_C(c))\left(\varepsilon_C\right)=\varepsilon_C(c)$$ for any $c\in C$, where in $(\dag)$ we identified $\Bbbk^*$ with $\Bbbk$. Hence, $\epsilon_C$ is comultiplicative and counital.

On the other hand, the injection $j_{A}:A^{\bullet }\hookrightarrow A^{\ast }$ induces a map $\eta_A:A\rightarrow A^{\bullet \ast }$ given by
\begin{equation}\label{eq:unitdef}
\eta_A\coloneqq j_A^*\circ \chi_A.
\end{equation}

We claim that this is an algebra morphism. Set $\eta=\eta_A$ for shortness. Then,  for all $a,b\in A$ and for any $f\in A^{\bullet }$:
\begin{equation*}
\begin{split}
m_{A^{\bullet \ast }}(\eta(a)\otimes \eta(b))(f) &\,=\,\varphi _{A^{\bullet },\, A^{\bullet }}\big(\eta(a)\otimes \eta(b)\big)\left(\Delta _{A^{\bullet }}(f)\right) \,=\, \sum \eta(a)\left( f_{1}\right) \eta(b)\left( f_{2}\right) \\
&\,=\,\sum f_{1}(a)f_{2}(b)\overset{(\ref{comultiplication})}{=}f(ab) \,=\,\eta(ab)(f).
\end{split}
\end{equation*}

Moreover, $\eta(u(1))(f)=f(1_{A})=\varepsilon _{A^{\bullet}}(f)=(\varepsilon _{A^{\bullet }})^{\ast }(1)(f)=u_{A^{\bullet \ast }}(1)(f)$, for all $f\in A^{\bullet }$.

Let us check finally that $\eta$ and $\epsilon$ satisfy the conditions to be the unit and the counit of the adjunction, respectively. By a direct calculation:
\begin{equation*}
\begin{split}
j_A\circ \eta_A^\bullet\circ\epsilon_{A^\bullet} & \stackrel{\eqref{eq:pallino}}{=} \eta_A^*\circ j_{A^{\bullet *}}\circ\epsilon_{A^\bullet}
 \stackrel{\eqref{eq:jC}}{=} \eta_A^*\circ \chi_{A^\bullet}
 \stackrel{\eqref{eq:unitdef}}{=} \chi_A^*\circ j_{A}^{**}\circ \chi_{A^\bullet}
 \stackrel{(*)}{=} \chi_A^*\circ \chi_{A^*}\circ j_{A} = j_A
\end{split}
\end{equation*}
where in $(*)$ we used the naturality of $j$ and the last equality follows from the fact the $(-)^*$ is adjoint to itself at the level of vector spaces. Therefore, by injectivity of $j_A$, we have that $\eta_A^\bullet\circ\epsilon_{A^\bullet}=\id{A^\bullet}$. For the other composition, let us compute
\begin{equation*}
\epsilon_{C}^*\circ \eta_{C^*} \stackrel{\eqref{eq:unitdef}}{=}  \epsilon_{C}^*\circ j_{C^*}^* \circ \chi_{C^*} \stackrel{\eqref{eq:jC}}{=} \chi_C^*\circ \chi_{C^*} = \id{C^*}
\end{equation*}
and this concludes the proof.
\end{proof}

\subsection{The  functor $(-)^{\bullet}$ and the tensor product}\label{ssec:mb}
Next we study how the functor $(-)^{\bullet}$ behaves with respect to the tensor product of two algebras.

\begin{proposition}\label{prop:barecela}
Let $A$ and $B$ be in $\nalgk$. Then the canonical injection $\varphi _{A, \,B}:A^{\ast }\otimes
B^{\ast }\longrightarrow (A\otimes B)^{\ast }$ of equation \eqref{Eq:varphi} induces the natural isomorphism in $\ncoalgk$
\begin{equation}\label{Eq:barecela}
\varphi _{A,B}^{\prime }\coloneqq (\varphi _{A^{\bullet},B^{\bullet }}\circ (\eta _{A}\otimes \eta _{B}))^{\bullet }\circ \epsilon_{(A^{\bullet }\otimes B^{\bullet })}:A^{\bullet }\otimes B^{\bullet }\overset{\cong }{\longrightarrow }(A\otimes B)^{\bullet }.
\end{equation}
\end{proposition}

\begin{proof}
Observe that for any $C,D$ in $\ncoalgk$, $\varphi_{C,D}:C^*\otimes D^*\to (C\otimes D)^*$ is a morphism in $\nalgk$. Thus the morphism defined in equation \eqref{Eq:barecela} is well-defined. For all $f\in A^{\bullet },g\in B^{\bullet },a\in A,b\in B$ we have that
\begin{eqnarray*}
((\varphi _{A^{\bullet },B^{\bullet }}\circ (\eta _{A}\otimes \eta _{B}))^{\bullet}\circ \epsilon _{(A^{\bullet }\otimes B^{\bullet })})(f\otimes g)(a\otimes b)
&=&(\varphi _{A^{\bullet },B^{\bullet }})^{\bullet }(\epsilon _{(A^{\bullet }\otimes
B^{\bullet })}(f\otimes g))(\eta _{A}(a)\otimes \eta _{B}(b)) \\
&=&\epsilon _{(A^{\bullet }\otimes B^{\bullet })}(f\otimes g)(\varphi _{A^{\bullet},B^{\bullet }}(\eta _{A}(a)\otimes \eta _{B}(b))) \\
&=&\varphi _{A^{\bullet },B^{\bullet }}(\eta _{A}(a)\otimes \eta
_{B}(b))(f\otimes g) \\
&=&\eta _{A}(a)(f)\eta _{B}(b)(g)\overset{(\ref{eq:unitdef})}{=}f(a)g(b) \\
&=&\left( \varphi _{A,B}(f\otimes g)\right) (a\otimes b).
\end{eqnarray*}
Therefore $j_{A\otimes B}\circ \varphi_{A,B}'=\varphi_{A,B}\circ (j_A\otimes j_B)$ and in particular $\varphi_{A,B}'$ is injective. It remains to find an inverse for $\varphi_{A,B}'$. To this aim consider the algebra morphisms
\begin{equation*}
i_{A}:A\longrightarrow A\otimes B:a\longmapsto a\otimes 1\qquad \mathrm{ and } \qquad i_{B}:B\longrightarrow A\otimes B:b\longmapsto 1\otimes b.
\end{equation*}
Then we can consider the map
\begin{equation*}
\xymatrix@C=40pt{\psi _{A,B}\coloneqq  (A\otimes B)^{\bullet } \ar@{->}^-{\Delta _{(A\otimes B)^{\bullet }}}[r] &  (A\otimes B)^{\bullet }\otimes (A\otimes B)^{\bullet } \ar@{->}^-{(i_{A})^{\bullet }\otimes (i_{B})^{\bullet }}[r] & A^{\bullet }\otimes B^{\bullet} .}
\end{equation*}
Note that $\psi _{A,B}$ satisfies, for all $f\in (A\otimes B)^{\bullet }$, $a\in A$, $b\in B$:
\begin{eqnarray*}
\varphi _{A,B}^{\prime }(\psi _{A,B}(f))(a\otimes b) &\coloneqq & \varphi
_{A,B}^{\prime }(((i_{A})^{\bullet }\otimes (i_{B})^{\bullet })(\Delta
_{(A\otimes B)^{\bullet }}(f)))(a\otimes b)=\varphi _{A,B}^{\prime }\left(\sum f_{1}\otimes f_{2}\right) (i_{A}(a)\otimes i_{B}(b)) \\
&=&\sum f_{1}(a\otimes 1)f_{2}(1\otimes b)\overset{(\ref{comultiplication})}{=}f((a\otimes 1)\cdot (1\otimes b))=f(a\otimes b).
\end{eqnarray*}
This means that $\varphi _{A,B}^{\prime }$ is also surjective and hence,  \textit{a fortiori}, it is an isomorphism with inverse $\psi _{A,B}$.
\end{proof}

\begin{remark}\label{rem:diamond}
Let $A$ be in $\nalgk$. We can consider $A^\circ$ as defined in \eqref{eq:usualfinitedual}. It should be observed that in general $A^{\circ }$ is strictly contained in $A^{\bullet }$. To show this take $A=C^*$ for a non-coassociative coalgebra $C$ that is not locally finite, then $A^\bullet$ cannot be locally finite (since $\epsilon _{C}:C\longrightarrow C^{\ast \bullet }$ is injective). At the same time, in view of Remark \ref{rem:oclassic}, $A^{\circ }=\mathrm{Loc}(A^{\bullet})$ so that it is locally finite and hence it cannot coincide with $A^{\bullet}$. We now provide an example of a coalgebra which is not locally finite. Explicitly, consider $C=\Bbbk[X]$ the polynomial ring in the indeterminate $X$ endowed with the comultiplication given by
\begin{gather*}
\Delta(1)=1\otimes 1, \quad \Delta(X)=X\otimes 1+ 1\otimes X, \quad \Delta(X^n)=X^n\otimes 1 +1\otimes X^n + X^{n+1}\otimes X + X\otimes X^{n+1}, \quad n\geq2
\end{gather*}
and the counit given by $\varepsilon(X^n)=\delta_{n,0}$ for all $n\geq 0$. It is easy to check that $(C,\Delta,\varepsilon)$ belongs to $\ncoalgk$. Note that factoring out by the coideal $\Bbbk 1$ and denoting, for $n\geq 0$, by $x_n$ the class of $X^{n+1}$ in the quotient yields the Lie coalgebra $E$ considered in \cite[page 9]{Michaelis-LieCoalg}. As for $E$, one easily proves that $X^2$ does not lie in any finite-dimensional subcoalgebra of $C$. Thus $C$ is not locally finite.
%
%Let $(A,m,u)$ be in $\nalgk$. Then the quadruple $(A,m,\id{A},u)$ defines a unital Hom-algebra and, as in \cite[section 2]
%{Elhamdadi},
\end{remark}

\section{Contravariant adjunction between the categories $\Nalgk$ and $\Ncoalgk$.}\label{sec:calgmu}

In this section we recall the definitions of the category of coalgebras with multiplication and unit and of that of algebras with comultiplication and counit, denoted by $\Nalgk$ and $\Ncoalgk$ respectively. For an object  $C \in \Nalgk$, we  investigate some basic properties of $C^{\bullet }$. The key result will be that, in this context, $C^\bullet$ is an algebra with comultiplication and counit. This will allow us to extend the contravariant adjunction of the previous section to a contravariant adjunction between the categories $\Nalgk$ and $\Ncoalgk$.

Let us start be recalling  from
\cite[Preliminaries]{MajidTannaka},  \cite[Definition XV.1.1 and Proposition XV.1.2]{Kassel}, the following definitions.

\begin{definition}\label{def:cmu} A \emph{coalgebra with multiplication and unit} is a datum $(C,\Delta,\varepsilon,m,u)$ where
\begin{enumerate}
\item the triple $(C,\Delta,\varepsilon)$ is in $\coalgk$;
\item the maps $m\colon C\otimes C\rightarrow C$ and $u\colon \K\rightarrow C$ are morphisms in $\coalgk$, called \emph{multiplication} and \emph{unit} respectively, such that $m$ is unital with unit $u$.
\end{enumerate}
\end{definition}
In other words this a \emph{non-necessarily associative monoid} (or algebra) inside the monoidal category of coassociative and counital coalgebras.
A morphism of coalgebras with multiplication and unit
is a linear map which is compatible with both structures, that is, simultaneously a morphism of coalgebras and of algebras.
The category  so obtained will be denoted by $\Nalgk$.

Dualizing Definition \ref{def:cmu} leads to the construction of the category of algebras with comultiplication and counit denoted by $\Ncoalgk$, whose objects are denoted by $(A,m,u,\Delta, \varepsilon)$. Thus, an object in $\Ncoalgk$ is a non necessarily coassociative comonoid inside the monoidal category of associative and unital  algebras.

\begin{proposition}\label{prop:FinDualCoquasi}
Let $\left( C,\Delta ,\varepsilon,m,u \right) $ be a coalgebra with multiplication and unit. Then $\left( C^{\bullet },\Delta^{\bullet},\varepsilon^{\bullet},m^{\bullet},u^{\bullet }\right) $ is an algebra with comultiplication and counit.  Moreover, this establishes a well-defined contravariant functor
$$
(-)^{\bullet}: \Nalgk \longrightarrow \Ncoalgk.
$$ %\vspace{-0.3cm}
\end{proposition}
\begin{proof}
For the reader's sake, in this proof we will write explicitly the isomorphism $\varphi_{C,\,D}':C^\bullet\otimes D^\bullet \to (C\otimes D)^\bullet$ of equation \eqref{Eq:barecela}, although in the statement we identified the domain and the codomain of this map. For this reason the multiplication of $C^\bullet$ is more precisely $m_{C^\bullet}\coloneqq \Delta^\bullet \circ \varphi_{C,\,C}'$ while its comultiplication is $\Delta_{C^\bullet}\coloneqq \left(\varphi_{C,\,C}'\right)^{-1}\circ m^\bullet$. We compute
\begin{equation*}
j_C\circ m_{C^\bullet}=j_C\circ \Delta^\bullet \circ  \varphi_{C,\,C}' \stackrel{\eqref{eq:pallino}}{=} \Delta^*\circ j_{C\otimes C} \circ \varphi_{C,\,C}' \stackrel{\eqref{Eq:barecela}}{=} \Delta^* \circ \varphi_{C,\,C} \circ (j_C\otimes j_C)=m_{C^*}\circ(j_C\otimes j_C).
\end{equation*}
Furthermore, since $\id{\Bbbk}\in \Bbbk^*=\Bbbk^\bullet$, we can compute $\varepsilon_C^\bullet(\id{\Bbbk})$ obtaining $\varepsilon_C=\varepsilon_C^\bullet(\id{\Bbbk})\in C^\bullet$. Thus we can set $1_{C^\bullet}\coloneqq\varepsilon_C$ and we have that $j_C(1_{C^\bullet})=1_{C^*}$. Since $j_C$ is injective, we deduce that $C^\bullet$ is an algebra. Explicitly, for every $f,g\in C^\bullet$ we have that $m_{C^\bullet}(f\otimes g)$ equals the convolution product $f*g$. Moreover $u^{\ast
}:C^{\ast }\rightarrow \Bbbk ^{\ast }\cong \Bbbk :f\mapsto f\left( 1\right) $
restricts to a map $\varepsilon _{C^{\bullet }}\coloneqq u^{\ast }:C^{\bullet
}\rightarrow \Bbbk $.

For all $x,y\in C$ and $f,g\in C^{\bullet }$, we have
\begin{eqnarray*}
\sum \left( f\ast g\right) _{1}\left( x\right) \cdot \left( f\ast g\right)
_{2}\left( y\right) &=&\left( f\ast g\right) \left( xy\right) =\sum f\left(
\left( xy\right) _{1}\right) \cdot g\left( \left( xy\right) _{2}\right) \\
&=&\sum f\left( x_{1}y_{1}\right) \cdot g\left( x_{2}y_{2}\right) =\sum
f_{1}\left( x_{1}\right) \cdot f_{2}\left( y_{1}\right) \cdot g_{1}\left(
x_{2}\right) \cdot g_{2}\left( y_{2}\right) \\
&=&\sum \left( f_{1}\ast g_{1}\right) \left( x\right) \cdot \left( f_{2}\ast
g_{2}\right) \left( y\right),
\end{eqnarray*}
and $\sum \varepsilon _{1}\left( x\right) \cdot \varepsilon _{2}\left( y\right)
=\varepsilon \left( xy\right) =\varepsilon \left( x\right) \cdot \varepsilon
\left( y\right)$. This implies that
\begin{equation*}
\sum \left( f\ast g\right) _{1}\otimes \left( f\ast g\right) _{2} =\sum
\left( f_{1}\ast g_{1}\right) \otimes \left( f_{2}\ast g_{2}\right) \qquad\mathrm{and}\qquad
\sum \varepsilon _{1}\otimes \varepsilon _{2} =\varepsilon \otimes
\varepsilon.
\end{equation*}
Thus $\Delta _{C^{\bullet }}$ is multiplicative and unital.  Moreover
\begin{equation*}
\varepsilon _{C^{\bullet }}\left( f\ast g\right) =\left( f\ast g\right)
\left( 1\right) =f\left( 1\right) \cdot g\left( 1\right) =\varepsilon
_{C^{\bullet }}\left( f\right) \cdot \varepsilon _{C^{\bullet }}\left( g\right)
\end{equation*}
and $\varepsilon _{C^{\bullet }}\left( \varepsilon \right) =\varepsilon \left(1\right) =1$ so that $\varepsilon _{C^{\bullet }}$ is multiplicative and unital as well.

Take a morphism  $f: C \to D$ in $\Nalgk$.  By Lemma  \ref{lemma:4.1},  we know that $f^{\bullet}:D^{\bullet} \to C^{\bullet} $ is a coalgebra map. It remains to check that it is multiplicative and unital. For every $x\in C$,
\begin{equation*}
f^\bullet\left(1_{D^\bullet}\right)\left(x\right)=
1_{D^\bullet}\big(f\left(x\right)\big)
=\varepsilon_D\big(f\left(x\right)\big)=
\varepsilon_C\left(x\right)=1_{C^\bullet}\left(x\right),
\end{equation*}
where in the third equality we used the fact that $f$ is a counital map. Furthermore, for every $\alpha$ and $\beta$ in $D^\bullet$
\begin{equation*}
f^\bullet\left(\alpha*\beta\right)\left(x\right) =\left(\alpha*\beta\right)\big(f(x)\big) =\sum \alpha\big(f\left(x\right)_1\big)\beta\big(f\left(x\right)_2\big) = \sum\alpha\big(f\left(x_1\right)\big)\beta\big(f\left(x_2\right)\big) = \big(f^\bullet\left(\alpha\right)*f^\bullet\left(\beta\right)\big)(x),
\end{equation*}
where we have used  the fact that  $f$ is comultiplicative.  This establishes the stated functor which is clearly a contravariant one.
\end{proof}

As we have mentioned in Remark \ref{rem:oclassic}, see also the references quoted therein, there is a functor
\begin{equation}\label{Eq:circ}
(-)^{\circ}: \algk \longrightarrow \coalgk,
\end{equation}
which is the restriction of the functor $(-)^{\bullet}$ defined in subsection \ref{ssec:fc}. Even thought we could use the same notation for  these two functors without ambiguity, we preferred  to keep different notations, in order to make clear the distinction   between the associative and the non-associative case.

\begin{lemma}\label{lemma:lift}
The functor $(-)^{\circ}$ is lifted  to a functor $(-)^{\circ}: \Ncoalgk \to \Nalgk$.  That is, we have a commutative diagram
$$
\xymatrix{\Ncoalgk \ar@{-->}^-{(-)^{\circ}} [rr] \ar@{->}^-{}[d] & & \Nalgk  \ar@{->}^-{}[d]  \\ \algk \ar@{->}^-{(-)^{\circ}}[rr] && \coalgk}
$$
where the vertical functors are the forgetful ones.
\end{lemma}
\begin{proof}
Take an object $(A, m,u,\Delta, \varepsilon)$ in $\Ncoalgk$. Analogously to subsection \ref{ssec:mb}, we can consider the datum $(A^{\circ}, m^{\circ}, u^{\circ}, \Delta^{\circ}, \varepsilon^{\circ})$, where $(A^{\circ}, m^{\circ}, u^{\circ})$, $\Delta^{\circ}$ and $\varepsilon^{\circ}$ are the images, up to the natural isomorphism of Proposition \ref{prop:barecela},  through the functor  \eqref{Eq:circ}  of  $(A, m,u)$, $\Delta$ and $\varepsilon$, respectively. We know that
$(A^{\circ}, m^{\circ}, u^{\circ})$ is coassociative counital coalgebra, and also that $\Delta^{\circ}$ and $\varepsilon^{\circ}$ are coalgebra maps. The unital property of $\Delta^{\circ}$ with respect to $\varepsilon^{\circ}$ is automatically derived from the counitality of $\Delta$ with respect to $\varepsilon$. Summing up, we have that  $(A^{\circ}, m^{\circ}, u^{\circ}, \Delta^{\circ}, \varepsilon^{\circ})$ is an object  the category $\Nalgk$.
To check that the stated functor is well-defined on morphisms one mimics the last part of the proof of Proposition \ref{prop:FinDualCoquasi}.
\end{proof}

The following is our first main result.

\begin{theorem}\label{th:finitedual}
There is a natural isomorphism:
$$
 \Ncoalgk\Big( A , C^{\bullet} \Big) \,\, \cong \,\,   \Nalgk\Big(C , A^{\circ}\Big)
$$
for every pair of objects $A$ in $\Ncoalgk$ and $C$ in $\Nalgk$.
That is, there is a contravariant adjunction
$$
\xymatrix{ \ar@<+0.9ex>@{->}^-{(-)^{\bullet}}[rr] \Nalgk & & \ar@{->}^-{(-)^{\circ}}[ll]  \Ncoalgk }
$$
where the contravariant functor $(-)^{\circ}$ is defined by Lemma \ref{lemma:lift} and $(-)^{\bullet}$ by Proposition \ref{prop:FinDualCoquasi}.
\end{theorem}
\begin{proof}
In order to prove the theorem let us consider the unit $\eta_A:A\to A^{\circ*}$ and the counit $\epsilon_C:C\to C^{*\circ}$ defined in \eqref{eq:unitdef} and \eqref{eq:jC} respectively, where $A$ is in $\Ncoalgk$ and $C$ is in $\Nalgk$.

We already know that $\eta_A$ is multiplicative and unital. We claim  that it lands into $A^{\circ\bullet}$, indeed let us show that $\textrm{Im}(\eta_A)$ is a good subspace of $A^{\circ *}$. For all $a\in A$ and $f,g\in A^\circ$ we have that
\begin{equation*}
\begin{split}
m^{\circ *}\left(\eta_A(a)\right)(f\otimes g) & = \eta_A(a)(f*g)=(f*g)(a) = \sum f(a_1)g(a_2) \\
 & = \sum \eta_A(a_1)(f)\eta_A(a_2)(g) = \left(\sum \eta_A(a_1)\otimes\eta_A(a_2)\right)(f\otimes g)
\end{split}
\end{equation*}
i.e.~for all $a\in A$, $m^{\circ *}(\eta_A(a)) \in \varphi_{A^\circ,\,A^\circ}\left(\textrm{Im}(\eta_A)\otimes \textrm{Im}(\eta_A)\right)$ where $\varphi_{-,-}$ is the canonical inclusion of equation \eqref{Eq:varphi}. Therefore we denote by $\xi_A:A\to A^{\circ\bullet}$ the corestriction of $\eta_A$. Observe further that, in particular, this means that $\eta_A$ (and hence $\xi_A$) is comultiplicative. Moreover, $\xi_A$ is also counital since
\begin{equation*}
u^\circ\left(\xi_A(a)\right)=\xi_A(a)\left(1_{A^\circ}\right)=\xi_A(a)\left(\varepsilon\right)=\varepsilon(a)
\end{equation*}
for all $a\in A$. By the foregoing,  $\xi_A$ is a morphism in the category $\Ncoalgk$. Now we can check the naturality in $A$ of $\xi_A$: pick a morphism $f\colon A\rightarrow B$ in $\Ncoalgk$ and consider the diagram
\begin{equation*}
\xymatrix@C=40pt{
A \ar[r]^-{\xi_A} \ar[d]_-{f} & A^{\circ \bullet} \ar[d]^-{f^{\circ \bullet}} \ar[r]^-{j_{A^\circ}} & A^{\circ *} \ar[d]^-{f^{\circ *}}  \\
B \ar[r]_-{\xi_B} & B^{\circ \bullet} \ar[r]_-{j_{B^\circ}} & B^{\circ *}
}
\end{equation*}
The commutativity of the outer diagram encodes the naturality of $\eta$, while the right-hand side diagram follows by \eqref{eq:pallino}. Hence the left-hand side diagram commutes too whence the naturality of $\xi$ is settled.
To construct the counit one proceeds in a very similar way. Explicitly, for an object $(C,\Delta,\varepsilon,m,u)$ in $\Nalgk$, the map $\epsilon_C$ induces the counit which is given by
\begin{equation}\label{Eq:PaoloWoutgelato}
\lfun{\vartheta_C}{C}{C^{\bullet\circ }}{x}{\Big[C^\bullet\rightarrow \K,\, g\mapsto g(x)\Big]}.
\end{equation}

It remains to check  the commutativity of the following two diagrams
\begin{equation*}
\xymatrix@C=40pt{
C^{\bullet \circ \bullet } \ar[r]^-{\left(\vartheta_C\right)^\bullet} & C^\bullet \\
C^\bullet \ar[u]^-{\eta_{C^\bullet}} \ar[ur]_-{=} &
} \qquad  \xymatrix@C=40pt{A^{\circ\bullet\circ} \ar[d]_-{\left(\eta_A\right)^\circ} & A^\circ \ar[l]_-{\vartheta_{A^\circ}} \ar[dl]^-{=} \\
 A^\circ  & }
\end{equation*}
As for the first one, for every $g\in C^\bullet$ and for every $c\in C$ a direct calculation shows that
\begin{equation*}
\left(\vartheta_C\right)^\bullet\left(\eta_{C^\bullet}\left(g\right)\right)(c)=\eta_{C^\bullet}\left(g\right)\left(\vartheta_C(c)\right)=\vartheta_C(c)(g)=g(c)
\end{equation*}
while for the second one, for every $f\in A^\circ$ and for every $a\in A$,
\begin{equation*}
\left(\xi_A\right)^\circ\left(\vartheta_{A^\circ}\left(f\right)\right)(a)=\vartheta_{A^\circ}\left(f\right)\left(\xi_A(a)\right)=\xi_A(a)\left(f\right)=f(a).
\end{equation*}
This finishes the proof.
\end{proof}

\section{Contravariant adjunction between quasi-bialgebras and split dual quasi-bialgebras.}\label{sec:qcadj}

This section contains our  main result. We show that the contravariant adjunction of Theorem \ref{th:finitedual} can be restricted to an adjunction between the category of quasi-bialgebras and that of split dual quasi-bialgebras. The later is a full subcategory of the category of dual quasi-bialgebras characterized by the fact that the 3-cocycle (i.e. the reassociator) splits as a finite sum of tensor products of linear maps.

\subsection{Quasi-bialgebras and dual quasi-bialgebras: Definitions and examples.}\label{ssec:qbalg}
We start by recalling the definition of  the main objects of this section. The definitions presented here are quoted form  \cite{Dri, MajidTannaka, Kassel}.
\begin{definition}[\cite{Dri}]\label{def:qbialg}
A \emph{quasi-bialgebra} is an object $(H,  m, u,\Delta, \varepsilon)$ in the category $\Ncoalgk$ (see after Definition \ref{def:cmu}), endowed with a counital 3-cocycle $\Phi$, i.e. an invertible element in the algebra $H\otimes H\otimes H$ that satisfies
\begin{gather}
\left(H\otimes H\otimes \Delta\right)(\Phi)\cdot \left(\Delta\otimes H\otimes H\right)(\Phi)=(1\otimes \Phi)\cdot (H\otimes \Delta\otimes H)(\Phi)\cdot(\Phi\otimes 1), \label{qb3} \\
(\varepsilon\otimes H\otimes H)(\Phi)\,=\, (H\otimes \varepsilon\otimes H)(\Phi)\,=\,(H\otimes H\otimes \varepsilon)(\Phi)\,=\, 1\otimes 1  , \label{qb4} \\
\Phi\cdot (\Delta\otimes H)(\Delta(h))=(H\otimes \Delta)(\Delta(h))\cdot \Phi. \label{qb2}
\end{gather}
The element $\Phi$ is called the \emph{reassociator}\footnote{In \cite[page 369]{Kassel} this is called the Drinfeld associator.} of the quasi-bialgebra. Obviously, if $\Phi=1 \tensor{} 1\tensor{}1$, then $(H,m,u,\Delta,\varepsilon)$ is a usual bialgebra.
\end{definition}

A linear map $\fk{f}: \left(H, m, u,\Delta, \varepsilon,  \Phi\right) \to \left(H', m', u',\Delta',\varepsilon', \Phi'\right)$ is a \emph{morphism of quasi-bialgebras} if it is a morphism in $\Ncoalgk$ such that
\begin{equation}\label{eq:QuasiMorphism}
\left(\fk{f}\otimes \fk{f}\otimes \fk{f}\right)\left(\Phi\right)=\Phi'.
\end{equation}
The category of quasi-bialgebras and their morphisms will be denoted by $\qbialgk$.

\begin{example}\label{exam:qbial}
Let $C$ be in $\ncoalgk$ and consider the tensor algebra $T(C)$. By the universal property of the tensor algebra, the comultiplication and the counit of $C$ induce a comultiplication and a counit on $T=T(C)$ respectively that make it into an object in $\Ncoalgk$. Suppose that $T$ is in $\qbialgk$. Then it admits a reassociator $\Phi\in T^{\otimes 3}$ but, in view of Corollary \ref{corollario:tensorn}, $\Phi\in \Bbbk\cdot 1\otimes 1\otimes 1$. By \eqref{qb4}, $\Phi=1\otimes 1\otimes 1$ which means that $T$ is in $\bialgk$. This forces $C$ to be in $\coalgk$. Therefore, if we consider $C$ in $\ncoalgk$ but not in $\coalgk$, then $T(C)$ is in $\Ncoalgk$ but not in $\qbialgk$.

We now give an explicit example of such a $C$. Consider the alternative algebra $A$, see Example \ref{exam:alt},  constructed as follows. As a vector space, $A=\K e \oplus \K x\oplus \K y$, with multiplication table given as follows:

\begin{table}[h]
\begin{tabular}{c|c|c|c}
$\cdot$ & $e$ & $x$ & $y$ \\
\hline
$e$ & $e$ & $x$ & $y$ \\
\hline
$x$ & $x$ & $y$ & $x$ \\
\hline
$y$ & $y$ & $x$ & $x$
\end{tabular}
\end{table}

This is a unital, commutative but not associative algebra. Let us take its ordinary linear dual $C=A^\ast=\K E \oplus \K X \oplus \K Y$ where $\left\{E,X,Y\right\}$ is the dual basis. It comes out to be a counital, cocommutative but not coassociative coalgebra. The induced comultiplication and counit are given by
\begin{gather*}
\Delta(X) = X\otimes Y+ Y\otimes X + Y\otimes Y + X\otimes E+E\otimes X, \qquad \varepsilon(X)=0, \\
\Delta(Y) = X\otimes X + Y\otimes E + E\otimes Y, \qquad \varepsilon(Y)=0, \\
\Delta(E) = E\otimes E, \qquad \varepsilon(E)=1.
\end{gather*}

Observe further that cocommutativity ensures that $T(C)$ satisfies the condition to be a coalternative coalgebra, as given explicitly in Example \ref{exam:alt} equation \eqref{Eq:alt}.
\end{example}

Besides algebras with comultiplication and counit that are not quasi-bialgebras, we also have quasi-bialgebras that are not bialgebras.

\begin{example}\label{exam:qnbial}
Let us retrieve a couple of examples.
First we exhibit a way to produce non-trivial quasi-bialgebras from ordinary bialgebras. Following \cite{Dri}, we consider a quasi-bialgebra $(H,m,u,\Delta,\varepsilon,\Phi)$ and we recall that a \emph{twist} on $H$ (also referred to as \emph{gauge transformation}, cf. \cite[Definition XV.3.1]{Kassel}) is an invertible element $F$ of $H\otimes H$ such that
\begin{equation*}
(H\otimes \varepsilon)(F)=1=(\varepsilon\otimes H)(F).
\end{equation*}
Given a twist, we can construct the twisted quasi-bialgebra $H_F\coloneqq (H,m,u,\Delta_F,\varepsilon, \Phi_F)$ where
\begin{equation*}
\Delta_F(a)\coloneqq F\cdot \Delta(a)\cdot F^{-1} \qquad \mathrm{and} \qquad \Phi_F \coloneqq  (1\otimes F)\cdot (A\otimes \Delta)(F)\cdot \Phi\cdot (\Delta\otimes A)(F^{-1})\cdot (F^{-1}\otimes 1).
\end{equation*}
This is still a quasi-bialgebra (cf.~\cite[Proposition XV.3.2]{Kassel}, \cite[page 1422]{Dri}). Now, assume we have $(G,m,u,\Delta,\varepsilon)$ an ordinary bialgebra. We can endow it with a trivial structure of quasi-bialgebra by considering $\Phi=1\otimes 1\otimes 1$. If we take a non trivial twist $F$ on $G$, then $G_F$ is a quasi-bialgebra, but it is not necessarily an ordinary bialgebra. Indeed:
$$\Phi_F=(1\otimes F)\cdot (G\otimes \Delta)(F)\cdot (\Delta\otimes G)(F^{-1})\cdot (F^{-1}\otimes 1)$$
does not equal $1\otimes 1\otimes 1$ in general, and $\Delta_F$ is not coassociative. In such cases, $G_F$ is a non-trivial example of quasi-bialgebra.

Next, let us show a case in which an ordinary bialgebra can be endowed with a non-trivial structure of quasi-bialgebra without changing its underlying structure. This example comes from \cite[Preliminaries 2.3]{EG} (see also \cite[Example 2.5]{BCT}). Let $C_2=\left\langle g\right\rangle$ be the cyclic group of order 2 with generator $g$ and let $\K$ be a field of characteristic different from 2. Consider the group algebra $H(2)\coloneqq \K C_2$ with its ordinary bialgebra structure, i.e., $\Delta(g)=g\otimes g$ and $ \varepsilon(g)=1$. Observe that $H(2)$ is  a two-dimensional commutative algebra. Now, set
$p\coloneqq \frac{1}{2}(1-g)$
and consider the non-trivial reassociator
$$\Psi\coloneqq (1\otimes 1\otimes 1)-2(p\otimes p\otimes p).$$
It can be easily verified that $(H(2),m,u,\Delta,\varepsilon, \Psi)$ satisfies the conditions to be a quasi-bialgebra.

Actually this example can tell us more: $H(2)$ with this non-trivial quasi-bialgebra structure turns out to be not twist equivalent to any ordinary bialgebra (by \emph{twist equivalent} to a bialgebra $G'$ we mean that there exists a twist $F$ on $H(2)$ and an isomorphism of quasi-bialgebras $G'\cong H(2)_F$; cf. \cite[page 1422]{Dri}). Hence this is a genuine example of a quasi-bialgebra. To see this note that $H(2)$ can be endowed with a \emph{quasi-antipode}, i.e. a triple $(s,\alpha,\beta)$ composed by an algebra anti-homomorphism $s:H(2)\rightarrow H(2)$ and two distinguished elements $\alpha$ and $\beta$ such that
\begin{equation*}
\sum s(a_1)\alpha a_2=\varepsilon(a)\,\alpha, \quad \sum a_1\beta s(a_2)=\varepsilon(a)\,\beta, \quad \sum \Phi^1\beta s(\Phi^2)\alpha\Phi^3=1, \quad \sum s(\phi^1)\alpha\phi^2\beta s(\phi^3)=1.
\end{equation*}
where $\sum \Phi^1\otimes \Phi^2\otimes \Phi^3\coloneqq \Phi$ and $\sum \phi^1\otimes \phi^2\otimes \phi^3\coloneqq \Phi^{-1}$ (a quasi-bialgebra with quasi-antipode is usually called a \emph{quasi-Hopf algebra}; cf. \cite[page 1424]{Dri}). In particular, $H(2)$ can be endowed with the quasi-antipode $(\id{H(2)},g,1)$. By \cite[page 1425]{Dri}, it is possible to twist a quasi-Hopf algebra too, and get again a quasi-Hopf algebra. As a consequence, if there were a twist $F$ on $H(2)$ such that $H(2)_F\cong G'$ where $G'$ is an ordinary bialgebra, then $H(2)_F$ would turn out to be an ordinary Hopf algebra. In particular $\beta_F\cdot \alpha_F=1$. However, writing $F\coloneqq \sum F^1\otimes F^2$, $F^{-1}\coloneqq \sum f^1\otimes f^2$ and recalling that $H(2)$ is commutative we get
\begin{equation*}
\beta_F\coloneqq \sum F^1\cdot \beta\cdot  s(F^2)=\sum F^1\cdot F^2, \qquad \alpha_F\coloneqq \sum s(f^1)\cdot \alpha\cdot  f^2=\sum f^1\cdot f^2 \cdot g,
\end{equation*}
and
\begin{equation*}
\beta_F\cdot \alpha_F=\left(\sum F^1\cdot F^2\right)\cdot \left(\sum f^1\cdot f^2 \cdot g\right) = \sum F^1\cdot f^1\cdot F^2 \cdot f^2 \cdot g = g \neq 1.
\end{equation*}
Therefore, $H(2)$ cannot be twist equivalent to any ordinary bialgebra.
\end{example}

For any  quasi-bialgebra $(H, m,u,\Delta, \varepsilon,\Phi)$, we denote by
$$
\Phi\coloneqq  \sum \Phi^1\tensor{}\Phi^2\tensor{}\Phi^3\,=\,\sum \Psi^1\tensor{}\Psi^2\tensor{}\Psi^3\,=\,\sum \Theta^1\tensor{}\Theta^2\tensor{}\Theta^3,
$$
the reassociator $\Phi$ of  Definition \ref{def:qbialg}, whose inverse is
$$
\phi: =\sum \phi^1\tensor{}\phi^2\tensor{}\phi^3\,=\, \sum \psi^1\tensor{}\psi^2\tensor{}\psi^3\,=\, \sum \theta^1\tensor{}\theta^2\tensor{}\theta^3.
$$
This notations will be soon understood.  Explicitly, equations \eqref{qb3}, \eqref{qb4} and \eqref{qb2}  can be rewritten as
\begin{eqnarray*}
\sum\Phi^1\cdot\Psi^1_{\;\Sscript{1}}\tensor{}\Phi^2 \cdot\Psi^1_{\;\Sscript{2}} \tensor{} \Phi^3_{\;\Sscript{1}} \cdot\Psi^2\tensor{}  \Phi^3_{\;\Sscript{2}}\cdot\Psi^3  &=&  \sum\Psi^1\cdot\Theta^1\tensor{}\Phi^1\cdot\Psi^2_{\;\Sscript{1}}\cdot\Theta^2 \tensor{} \Phi^2\cdot\Psi^2_{\;\Sscript{2}}\cdot\Theta^3\tensor{} \Phi^3 \cdot\Psi^3,  \\
\sum\Phi^1 \varepsilon(\Phi^2) \tensor{}\Phi^3 &=& 1\tensor{}1, \\
\sum\Phi^1 h_{\Sscript{1,\, 1}} \tensor{}  \Phi^2 h_{\Sscript{1,\, 2}} \tensor{} \Phi^3 h_{\Sscript{2}} &=& \sum h_{\Sscript{1}}\Phi^1\tensor{}   h_{\Sscript{2,\, 1}}\Phi^2 \tensor{}  h_{\Sscript{2,\, 2}}\Phi^3. \\
\end{eqnarray*}

The dual version of Definition \ref{def:qbialg}, led in \cite{MajidTannaka}  to the notion of \emph{dual quasi-bialgebra}, which, for sake of reader convenience, we recall here with details.

\begin{definition}\label{def:dqb} A \emph{dual quasi-bialgebra} is  an object $(U,\Delta, \varepsilon, m, u)$ in the category $\Nalgk$, endowed with a unital 3-cocycle $\omega$, i.e. a convolution invertible element $\omega \colon U\otimes U\otimes U\rightarrow \K$ that satisfies
\begin{eqnarray}
\Big(\omega \circ (U\otimes U\otimes m) \Big) \ast \Big( \omega \circ \left( m\otimes
U\otimes U\right) \Big) & = & \big( \varepsilon \otimes \omega \big) \,\ast \, \Big( \omega
\circ \left( U\otimes m\otimes U\right)\Big) \, \ast\, \big( \omega \otimes \varepsilon \big)  \label{dqb3} \\
\omega \left( h\otimes k\otimes l\right) & = & \varepsilon
\left( h\right) \varepsilon \left( k\right) \varepsilon \left( l\right),\;
\text{ whenever }\,1_{U}\in \{h,k,l\}\, \subset U \label{dqb4} \\
\Big(u \circ \omega \Big)\ast \Big( m \circ \left( m\otimes U\right) \Big) & = & \Big(m \circ \left( U\otimes m\right)\Big) \ast \Big(u \circ \omega \Big),\label{dqb2}
\end{eqnarray}
where the star $*$ in  equation \eqref{dqb3}  stands for the convolution product of the algebra $(U^{\Sscript{\otimes 4}})^*$, while in equation \eqref{dqb2} it is the convolution product of the non-associative algebra $\vectk\big(U^{\Sscript{\otimes 3}}, U \big)$.
The map $\omega $ is  also called  the \emph{reassociator} of the dual quasi-bialgebra (this is an invertible element in the convolution algebra $(U^{\Sscript{\otimes 3}})^*$).
\end{definition}
A linear map $\varfun{\fk{g}}{\left(U,m, u,\Delta, \varepsilon,\omega\right)}{\left(U',m', u',\Delta', \varepsilon',\omega'\right)}$ is a morphism of dual quasi-bialgebras if it is a morphism in the category $\Nalgk$  satisfying:
\begin{equation}\label{eq:DualQuasiMorphism}
\omega'\circ\left(\fk{g}\otimes \fk{g}\otimes \fk{g}\right)=\omega.
\end{equation}
The category of dual quasi-bialgebras and their morphisms will be denoted by $\dqbialgk$.

On elements, the equations  \eqref{dqb3}, \eqref{dqb4} and \eqref{dqb2}  are written, for all $x,y,z,t\in U$, as
\begin{gather*}
\sum\omega\big( x_{\Sscript{1}}\otimes y_{\Sscript{1}} \otimes z_{\Sscript{1}}t_{\Sscript{1}}\big)\omega\big(x_{\Sscript{2}} y_{\Sscript{2}}\otimes z_{\Sscript{2}}\otimes t_{\Sscript{2}}\big) =  \sum\omega\big( y_{\Sscript{1}}\otimes z_{\Sscript{1}}\otimes t_{\Sscript{1}}\big)\omega\big(x_{\Sscript{1}}\otimes y_{\Sscript{2}} z_{\Sscript{2}}\otimes t_{\Sscript{2}}\big) \omega\big(x_{\Sscript{2}}\otimes y_{\Sscript{3}}\otimes z_{\Sscript{3}}\big), \\
\omega(x\otimes y\otimes 1)=\omega(x\otimes 1\otimes y)=\omega(1\otimes x\otimes y) = \varepsilon(x) \varepsilon(y), \\
\sum \omega\big( x_{\Sscript{1}} \otimes   y_{\Sscript{1}} \otimes z_{\Sscript{1}}\big) \big( x_{\Sscript{2}}   y_{\Sscript{2}}\big) z_{\Sscript{2}} =  \sum x_{\Sscript{1}}\big(y_{\Sscript{1}} z_{\Sscript{1}}\big) \omega\big(  x_{\Sscript{2}} \otimes y_{\Sscript{2}}\otimes  z_{\Sscript{2}} \big).
\end{gather*}

\subsection{The contravariant adjunctions.} We first check that the functor of Lemma \ref{lemma:lift}, leads to a functor from the category $\qbialgk$ to $\dqbialgk$. Consider  a quasi-bialgebra $(H, m, u,\Delta, \varepsilon, \Phi)$. Since the  underlying object $(H, m, u,\Delta, \varepsilon)$ is in $\Ncoalgk$, we can consider its image  $(H^{\circ}, m^{\circ}, u^{\circ},\Delta^{\circ}, \varepsilon^{\circ})$ in $\Nalgk$ by the functor  of Lemma \ref{lemma:lift}.  Set  $U=H^{\circ}$ and consider the natural transformation of \eqref{eq:unitdef} at $H^{\Sscript{ \tensor{} 3}}$
$$
\xymatrix@C=40pt{ H^{\Sscript{\otimes 3}}  \, \ar@{->}^-{\eta_{\Sscript{H^{\tensor{} 3}}}}[r] & \Big(\big( H^{\Sscript{\otimes 3}}\big)^{\circ}\Big)^* \, \overset{\eqref{Eq:barecela}}{\cong} \, \Big( \big( H^{\circ}\big)^{\Sscript{\tensor{} 3}}\Big)^*\,=\, \Big( U^{\Sscript{\tensor{} 3}}\Big)^*,}
$$
which by construction is an algebra map. Therefore, the following  $\Bbbk$-linear map
\begin{equation}\label{Eq:omega}
\omega\coloneqq  \eta_{\Sscript{H^{\tensor{} 3}}} (\Phi): U^{\Sscript{\tensor{}3}} \longrightarrow \Bbbk,\quad \Big( f\tensor{}g \tensor{}h\longmapsto  \sum f(\Phi^1) g(\Phi^2) h(\Phi^3) \Big),
\end{equation}
is an invertible element in the convolution algebra $\big( U^{\Sscript{\tensor{} 3}}\big)^*$, since $\Phi$ is so in the algebra $H^{\Sscript{\tensor{} 3}}$.

We claim that $(U, m^{\circ}, u^{\circ},\Delta^{\circ}, \varepsilon^{\circ}, \omega)$ is now a dual quasi-bialgbra.
Taken an element $x \in H$, we can compute
\begin{eqnarray*}
\Big( \sum \omega(f_{\Sscript{1}}\otimes g_{\Sscript{1}}\otimes h_{\Sscript{1}}\big) \big(f_{\Sscript{2}}  g_{\Sscript{2}}\big) h_{\Sscript{2}}\Big)  (x) &=& \sum f_{\Sscript{1}}(\Phi^1) g_{\Sscript{1}}(\Phi^2) h_{\Sscript{1}}(\Phi^3) f_{\Sscript{2}}(x_{\Sscript{1,\, 1}})g_{\Sscript{2}}(x_{\Sscript{1,\, 2}}) h_{\Sscript{2}}(x_{\Sscript{2}})\\
&\overset{\eqref{qb2}}{=}& \sum f_{\Sscript{1}}(\Phi^1) g_{\Sscript{1}}(\Phi^2) h_{\Sscript{1}}(\Phi^3) f_{\Sscript{2}}( \phi^1 x_{\Sscript{1}}\Phi^1)g_{\Sscript{2}}(\phi^2x_{\Sscript{2,\, 1}}\Phi^2) h_{\Sscript{2}}(\phi^3x_{\Sscript{2,\,2}}\Phi^3) \\
&\overset{\eqref{comultiplication}}{=}&  \sum f(\Phi^1 \phi^1x_{\Sscript{1}}\Phi^1) g(\Phi^2\phi^2 x_{\Sscript{2,\, 1}}\Phi^2) h(\Phi^3\phi^3 x_{\Sscript{2,\,2}}\Phi^3) \\ &=&   \sum f( x_{\Sscript{1}}\Phi^1)g(x_{\Sscript{2,\, 1}}\Phi^2) h(x_{\Sscript{2,\,2}}\Phi^3)  \\
&\overset{\eqref{comultiplication}}{=}& \sum f_{\Sscript{1}}(x_{\Sscript{1}}) g_{\Sscript{1}}(x_{\Sscript{2,\,1}}) h_{\Sscript{1}}(x_{\Sscript{2,\,2}}) f_{\Sscript{2}}( \Phi^1)g_{\Sscript{2}}(\Phi^2) h_{\Sscript{2}}(\Phi^3) \\ &=&  \sum  f_{\Sscript{1}}(x_{\Sscript{1}}) g_{\Sscript{1}}(x_{\Sscript{2,\,1}}) h_{\Sscript{1}}(x_{\Sscript{2,\,2}})\omega\big( f_{\Sscript{2}}\otimes g_{\Sscript{2}}\otimes h_{\Sscript{2}}\big)\\
&=& \Big( \sum f_{\Sscript{1}}\big(g_{\Sscript{1}} h_{\Sscript{1}} \big) \omega\big( f_{\Sscript{2}}\otimes g_{\Sscript{2}}\otimes  h_{\Sscript{2}}\big)\Big) (x)
\end{eqnarray*}
This gives equation \eqref{dqb2} for $(U,\omega)$. Equation \eqref{dqb4} for $(U,\omega)$, follows by:
\begin{equation*}
\omega(f\otimes1\otimes h) = f(\Phi^1)\varepsilon(\Phi^2) h(\Phi^3) \overset{\eqref{qb4}}{=} f(1)h(1) = \varepsilon(f)\varepsilon(h),
\end{equation*}
and similarly when $1$ appears in the other entries.

Let us check equation \eqref{dqb3} for $(U,\omega)$. Considered $f,g,h, e \in U$, we have
\begin{eqnarray*}
\lefteqn{\sum \omega\big( f_{\Sscript{1}}\otimes  g_{\Sscript{1}}\otimes h_{\Sscript{1}}e_{\Sscript{1}}\big)\omega\big(f_{\Sscript{2}} g_{\Sscript{2}}\otimes h_{\Sscript{2}}\otimes  e_{\Sscript{2}}\big) }
\qquad \qquad  \\ &=& \sum f_{\Sscript{1}}(\Phi^1) g_{\Sscript{1}}(\Phi^2) \big(h_{\Sscript{1}}e_{\Sscript{1}}\big)(\Phi^3)\, \big(f_{\Sscript{2}} g_{\Sscript{2}}\big)(\Psi^1) h_{\Sscript{2}}(\Psi^2) e_{\Sscript{2}}(\Psi^3)\\
&=& \sum f_{\Sscript{1}}(\Phi^1) g_{\Sscript{1}}(\Phi^2) h_{\Sscript{1}}(\Phi^3{}_{\Sscript{1}})e_{\Sscript{1}}(\Phi^3{}_{\Sscript{2}})\, f_{\Sscript{2}}(\Psi^1{}_{\Sscript{1}}) g_{\Sscript{2}}(\Psi^1{}_{\Sscript{2}}) h_{\Sscript{2}}(\Psi^2) e_{\Sscript{2}}(\Psi^3) \\
&=& \sum f\big(\Phi^1 \Psi^1{}_{\Sscript{1}}  \big) g\big(\Phi^2 \Psi^1{}_{\Sscript{2}}\big) h\big(\Phi^3{}_{\Sscript{1}} \Psi^2\big)e\big(\Phi^3{}_{\Sscript{2}} \Psi^3 \big)  \\
&\overset{\eqref{qb4}}{=}& \sum f\big(\Psi^1\Theta^1 \big) g\big(\Phi^1 \Psi^2{}_{\Sscript{1}} \Theta^2\big) h\big(\Phi^2 \Psi^2{}_{\Sscript{2}} \Theta^3 \big)e\big(\Phi^3 \Psi^3\big) \\
&=&   \sum  f_{\Sscript{1}}(\Psi^1)f_{\Sscript{2}}(\Theta^1)\, g_{\Sscript{1}}(\Phi^1) g_{\Sscript{2}}(\Psi^2{}_{\Sscript{1}}) g_{\Sscript{3}}(\Theta^2) \, h_{\Sscript{1}}(\Phi^2) h_{\Sscript{2}}(\Psi^2{}_{\Sscript{2}})  h_{\Sscript{3}}(\Theta^3)\,e_{\Sscript{1}}(\Phi^3) e_{\Sscript{2}}( \Psi^3) \\
&=& \sum \Big( g_{\Sscript{1}}(\Phi^1) h_{\Sscript{1}}(\Phi^2) e_{\Sscript{1}}(\Phi^3)\Big)\Big(f_{\Sscript{1}}(\Psi^1)g_{\Sscript{2}}(\Psi^2{}_{\Sscript{1}}) h_{\Sscript{2}}(\Psi^2{}_{\Sscript{2}}) e_{\Sscript{2}}( \Psi^3)\Big)  \Big(f_{\Sscript{2}}(\Theta^1) g_{\Sscript{3}}(\Theta^2)  h_{\Sscript{3}}(\Theta^3) \Big) \\
&=& \sum \Big( g_{\Sscript{1}}(\Phi^1) h_{\Sscript{1}}(\Phi^2) e_{\Sscript{1}}(\Phi^3)\Big)\Big(f_{\Sscript{1}}(\Psi^1)(g_{\Sscript{2}} h_{\Sscript{2}})(\Psi^2) e_{\Sscript{2}}( \Psi^3)\Big)  \Big(f_{\Sscript{2}}(\Theta^1) g_{\Sscript{3}}(\Theta^2)  h_{\Sscript{3}}(\Theta^3) \Big) \\
&=& \sum \omega\big( g_{\Sscript{1}}\otimes h_{\Sscript{1}}\otimes e_{\Sscript{1}}\big)\, \omega\big(f_{\Sscript{1}}\otimes g_{\Sscript{2}} h_{\Sscript{2}}\otimes e_{\Sscript{2}}\big)\, \omega\big(f_{\Sscript{2}}\otimes g_{\Sscript{3}}\otimes h_{\Sscript{3}} \Big),
\end{eqnarray*}
where we have used the convolution product and the formula \eqref{comultiplication}.  This completes the proof of the claim.

Furthermore, it is by definition  that any morphism $\fk{f}: (H,m,u,\Delta,\varepsilon,\Phi)  \longrightarrow (H',m',u',\Delta',\varepsilon',\Phi')$ of quasi-bialgebras,  is a morphism in the category $\Ncoalgk$. Then, by applying the functor $(-)^{\circ}$ of Lemma \ref{lemma:lift}, we get that  $\fk{f}^{\circ}: (H'{}^{\circ},\Delta'{}^{\circ},\varepsilon'{}^{\circ}, m'{}^{\circ},u'{}^{\circ}) \longrightarrow (H^{\circ}, \Delta^{\circ},\varepsilon^{\circ}, m^{\circ},u^{\circ})$ is a morphism in the category $\Nalgk$. Therefore, we only need to check the compatibility condition with reassociators constructed in equation \eqref{Eq:omega}, which is derived as follows:
\begin{eqnarray*}
\omega \Big(\left(\fk{f}^{\circ}\otimes
\fk{f}^{\circ}\otimes\fk{f}^{\circ}\right)\left(f\otimes g \otimes h\right)\Big) & = & \omega \Big(\left(f\circ\fk{f}\right)\otimes \left(g\circ\fk{f}\right)\otimes \left(h\circ\fk{f}\right)\Big) = \left(f\otimes g \otimes h\right)\Big(\left(\fk{f}\otimes \fk{f}\otimes \fk{f}\right)\left(\Phi\right)\Big) \\
 & \stackrel{\eqref{eq:QuasiMorphism}}{=} & \left(f\otimes g \otimes h\right)\left(\Phi'\right) =\omega'\left(f\otimes g \otimes h\right).
\end{eqnarray*}
Hence $\fk{f}$ satisfies \eqref{eq:DualQuasiMorphism} and it is a morphism of dual quasi-bialgebras. Then, we have established a contravariant functor
\begin{equation}\label{Eq:grandine}
(-)^{\circ}:  \qbialgk  \longrightarrow \dqbialgk,
\end{equation}
which obviously converts the following diagram
$$
\xymatrix@C=35pt{  \Ncoalgk \ar@{->}^-{(-)^\circ}[rr]  & & \Nalgk \\ \qbialgk \ar@{-->}^-{(-)^\circ}[rr] \ar@{->}^-{}[u] & & \dqbialgk  \ar@{->}^-{}[u]  }
$$
commutative, where the vertical functors are the canonical forgetful functors.

In the other way around, take an object $(U,\Delta,\varepsilon,m,u,\omega)$ in the category $\dqbialgk$.  Thus, we can take the image of its underlying object $(U,\Delta,\varepsilon,m,u)$ by the functor of Proposition \ref{prop:FinDualCoquasi}, that is the object  $(U^{\bullet},m^{\bullet},u^{\bullet},\Delta^{\bullet},\varepsilon^{\bullet})$ of the category $\Ncoalgk$. The problem now is to construct  a reassociator $\Phi$ for $U^{\bullet}$, i.e. a unital 3-cocycle. It seems that \emph{a priori}  there is no obvious way to deduce this cocycle directly from the starting datum $(U,\Delta,\varepsilon,m,u,\omega)$. To this aim, an assumption should be postulated. First, we consider the following natural transformation:
\begin{equation}\label{Eq:Gamma}
\xymatrix@C=40pt{\zeta:  \big(U^{\bullet}\big){}^{\Sscript{\otimes 3}}  \, \ar@{^{(}->}^-{}[r]^-{(j_U)^{\otimes 3}} & \big(U^*\big) ^{\Sscript{\otimes 3}} \,  \ar@{->}^-{\varphi_{\Sscript{U,\,U}}\tensor{}U^*}[r]  & \big( U\tensor{}U \big)^*\tensor{} U^* \ar@{->}^-{\varphi_{\Sscript{U\tensor{}U},\, U}}[r]   &  \big(U^{\Sscript{\tensor{} 3}}\big)^*, }
\end{equation}
which, up to the isomorphism $\big(U^{\Sscript{\otimes 3}}\big){}^{\bullet} \cong \big(U^{\bullet}\big){}^{\Sscript{\otimes 3}} $ of equation \ref{Eq:barecela},   coincides with the canonical injection of the total good subspace of  $\big(U^{\Sscript{\tensor{} 3}}\big)^*$. Notice that $\zeta$ is an algebra map, as it is a composition of algebra maps. Moreover, it is easily seen that $\zeta$ is in fact a natural transformation at $U$.

\begin{proposition}\label{prop:EssentialImage}
Let $(U,\Delta,\varepsilon,m,u,\omega)$ be a dual quasi-bialgebra. Assume  there exists an invertible element $\Phi\in \left(U^\bullet\right)^{\otimes 3}$ such that $\zeta\left(\Phi\right)=\omega$, then $(U^\bullet ,m^{\bullet},u^{\bullet},\Delta^{\bullet},\varepsilon^{\bullet}, \Phi)$ is a quasi-bialgebra.
\end{proposition}
\begin{proof}
Write $\Phi=\sum \Phi^1\tensor{}\Phi^2\tensor{}\Phi^3$.   Then $\omega(x\tensor{}y\tensor{}z)\,=\, \sum \Phi^1(x)\Phi^2(y)\Phi^3(z)$, for every $x, y, z \in U$. Using this equality, equations  \eqref{dqb3},\eqref{dqb4}) and \eqref{dqb2} are easily transferred to equations \eqref{qb3}, \eqref{qb4} and \eqref{qb2}, respectively. This concludes the proof.
\end{proof}

\begin{corollary}\label{cor:doubledual}
Let $(H,m,u,\Delta,\varepsilon,\Phi)$ be a quasi-bialgebra. Then $(H^{\circ\bullet},m^{\circ\bullet},u^{\circ\bullet},\Delta^{\circ\bullet},\varepsilon^{\circ\bullet})$ is still a quasi-bialgebra with reassociator $\Psi\coloneqq \left(\xi_H\right)^{\otimes 3}(\Phi)$, where $\xi$ is the unit of the adjunction of Theorem \ref{th:finitedual}.
\end{corollary}
\begin{proof}
We already know from \eqref{Eq:grandine} that $(H^{\circ},m^{\circ},u^{\circ},\Delta^{\circ},\varepsilon^{\circ},\omega)$ is a dual quasi-bialgebra with reassociator given by $\omega=\zeta\left(\left(\xi_H\right)^{\otimes 3}(\Phi)\right)$. Now apply Proposition \ref{prop:EssentialImage} to conclude.
\end{proof}

Let us denote by $\sdqbialgk$ the full subcategory of the category $\dqbialgk$ whose objects are \emph{split dual quasi-bialgebras}, i.e.~dual quasi-bialgebras $(U,\Delta,\varepsilon,m,u,\omega)$ such that there exists an invertible element $\Phi\in \left(U^\bullet\right)^{\otimes 3}$  with  $\zeta\left(\Phi\right)=\omega$. In this way the assignment described in Proposition \ref{prop:EssentialImage} yields the functor
\begin{equation}\label{Eq:Logic}
(-)^{\bullet}: \sdqbialgk \longrightarrow \qbialgk,
\end{equation}
acting by identity on morphisms.
We are led to the following  main result.

\begin{theorem}\label{them:main}
The contravariant adjunction of Theorem \ref{th:finitedual}
induces the  contravariant adjunction
$$
\xymatrix{ \ar@<+0.9ex>@{->}^-{(-)^{\bullet}}[rr] \sdqbialgk & & \ar@{->}^-{(-)^{\circ}}[ll]  \qbialgk }
$$
where the contravariant functor $(-)^{\circ}$ is the one of \eqref{Eq:grandine}, and $(-)^\bullet$ is the one of \eqref{Eq:Logic}.
\end{theorem}
\begin{proof}
The only thing we need to check is that the unit and the counit of the adjunction of Theorem \ref{th:finitedual} preserve the reassociator of a quasi-bialgebra and the one of a dual quasi-bialgebra respectively. For the unit, which is given by $$\xi : \id{\Ncoalgk} \longrightarrow  (-)^{\bullet} \circ (-)^{\circ}$$ as in the proof of Theorem \ref{th:finitedual},  this follows directly from Corollary \ref{cor:doubledual}.

As for the counit  $$\vartheta: \id{\Nalgk} \longrightarrow (-)^{\circ} \circ (-)^\bullet$$ which is given by \eqref{Eq:PaoloWoutgelato}, consider a dual quasi-bialgebra $(U,\Delta,\varepsilon,m,u,\omega)$ in $\sdqbialgk$; this means that there exists an element $\Phi\in \left(U^\bullet\right)^{\otimes 3}$ such that $\zeta(\Phi)=\omega$ and that $(U^\bullet,\Delta^\bullet,\varepsilon^\bullet,m^\bullet,u^\bullet,\Phi)$ is a quasi-bialgebra, where $\zeta$ is the natural transformation of \eqref{Eq:Gamma}.
From the definition of the functor in \eqref{Eq:grandine}, we have that the reassociator for the dual quasi-bialgebra $(U^{\bullet\circ},\Delta^{\bullet\circ},\varepsilon^{\bullet\circ},m^{\bullet\circ},u^{\bullet\circ})$ is clearly given by $\zeta\left(\left(\xi_{U^\bullet}\right)^{\otimes 3}(\Phi)\right)$. Now observe that the  following computation
\begin{eqnarray*}
\zeta\left(\big(\xi_{U^\bullet}\big)^{\otimes 3}(\Phi)\right)\circ\left(\vartheta_{U}\right)^{\otimes 3} & = & \left(\left(\vartheta_{U}\right)^{\otimes 3}\right)^\ast\left(\zeta\left(\big(\xi_{U^\bullet}\big)^{\otimes 3}(\Phi)\right)\right) \stackrel{(\textrm{nat. of }\zeta)}{=}  \zeta\left(\big(\left(\vartheta_{U}\right)^\bullet\big)^{\otimes 3}\left(\big(\xi_{U^\bullet}\big)^{\otimes 3}(\Phi)\right)\right) \\
 & = & \zeta\left(\big(\left(\vartheta_U\right)^\bullet\circ\xi_{U^\bullet}\big)^{\otimes 3}(\Phi)\right) = \zeta(\Phi)=\omega
\end{eqnarray*}
shows that $\vartheta$ preserves reassociators as desired. Hence, the unit comes out to be a quasi-bialgebra map and the counit a dual quasi-bialgebra map, settling the adjunction.
\end{proof}

\begin{remark}
Recall that a subcategory $\mathcal{B}$ of a category $\mathcal{A}$ is \emph{closed under sources} whenever for any morphism $f\colon a\to b$ in $\mathcal{A}$, if $b$ is in $\mathcal{B} $ then $a$ is in $\mathcal{B}$. Let us check that $\sdqbialgk$ is closed under sources when regarded as a subcategory of $\dqbialgk$. Let $\fk{g}\colon (U^\prime,\omega^\prime) \rightarrow (U,\omega)$ be a morphism in $\dqbialgk$ such that $(U,\omega)$ is an object in $\sdqbialgk$.  By assumption, there exists  $\Phi= \sum \Phi^1\otimes \Phi^2\otimes \Phi^3\in \big(U^\bullet\big)^{\otimes 3}$ such that $\omega =\zeta\left(\Phi\right)$. Since $\fk{g}$ preserves the reassociator, we have that
\begin{multline*}
\omega'= \omega\circ\left(\fk{g}^{\otimes 3}\right) = \zeta(\Phi)\circ\left(\fk{g}^{\otimes 3}\right)   = \varphi'{}_{\Sscript{U^\prime\otimes U^\prime,U^\prime}}\circ\left(\varphi'{}_{\Sscript{U^\prime,U^\prime}}\otimes U^{\prime\ast}\right)\Big(\left(\Phi^1\circ \fk{g}\right)\otimes \left(\Phi^2\circ \fk{g}\right)\otimes \left(\Phi^3\circ \fk{g}\right)\Big) =\zeta\left(\big(\fk{g}^{\bullet}\big)^{\otimes 3}\left(\Phi\right)\right),
\end{multline*}
where $\varphi'_{-,-}$ is the natural transformation of equation \eqref{Eq:barecela}. This means that
$\omega^\prime$ itself comes out to be the image by $\zeta$ of  $\big(\fk{g}^{\bullet}\big)^{\otimes 3}\left(\Phi\right)$ that lies in $\big(U^\prime\big)^\bullet$. Therefore,  $(U^\prime,\omega^\prime)$ belongs to $\sdqbialgk$.

Let us observe briefly that $\sdqbialgk$ is a proper subcategory of $\dqbialgk$: in fact the subsequent example exhibits a dual quasi-bialgebra whose reassociator does not split. This means moreover that this particular dual quasi-bialgebra cannot be the finite dual of a quasi-bialgebra (in view of the definition of the reassociator given in \eqref{Eq:omega}).
\end{remark}

\begin{example}
Let $\Bbbk $ be a field and consider $\Bbbk[X]$ the ring of
polynomials in one indeterminate $X$ with the monoid bialgebra structure,
i.e. $\Delta \left( X\right) =X\otimes X$,  $\varepsilon \left( X\right) =1$.
Let us consider a map $\varphi :\Bbbk [X]\longrightarrow \Bbbk $ not in $\Bbbk [X]^{\circ }$, the ordinary
finite dual of $\Bbbk [X]$ (which, in this case, coincides with $\Bbbk
[X]^{\bullet }$), and such that $\varphi(1)=1$, $\varphi \left( X^{n}\right) \neq 0$
for all $n\geq 1$. Let us build a 3-cocycle $\omega $ that does not split by mean of $\varphi $. Recalling that a basis for $\Bbbk [X]\otimes \Bbbk \lbrack
X]\otimes \Bbbk [X]$ is given by the elements $X^{n}\otimes
X^{k}\otimes X^{m}$ for $m,k,n\geq 0$, let us define $\omega $ on this basis
as follows, and then extend it by linearity. For all $m,n,k\geq 0$ let
us set:%
\begin{gather*}
\omega \left( 1\otimes X^{n}\otimes X^{m}\right) = \omega \left(
X^{n}\otimes 1\otimes X^{m}\right) = \omega \left( X^{n}\otimes
X^{m}\otimes 1\right)\coloneqq 1 ; \\
\omega \left( X^{n}\otimes X^{k+1}\otimes X^{m}\right) \coloneqq \varphi \left(
X^{k}\right) ^{-2}\varphi \left( X^{n+k}\right) \varphi \left(
X^{m+k}\right) .
\end{gather*}

Observe that the given comultiplication ensures that we have%
\begin{equation*}
\omega ^{-1}\left( X^{n}\otimes X^{k}\otimes X^{m}\right) =\omega \left(
X^{n}\otimes X^{k}\otimes X^{m}\right) ^{-1}=\frac{1}{\omega \left(
X^{n}\otimes X^{k}\otimes X^{m}\right) }
\end{equation*}%
for all $m,k,n\geq 0$. Now, let us show that $\omega$ is actually a unital
3-cocycle. It is unital by definition. If $0\in \{m,n,r,s\}$ then we trivially have
$$\omega \left( X^{m}\otimes X^{r}\otimes X^{s}\right) \omega \left(X^{n}\otimes X^{m+r}\otimes X^{s}\right) \omega \left( X^{n}\otimes X^{m}\otimes X^{r}\right)=\omega \left( X^{n}\otimes X^{m}\otimes X^{r+s}\right) \omega \left(X^{n+m}\otimes X^{r}\otimes X^{s}\right).$$
For all $m,n,r,s\geq 1$ we have
\begin{equation*}
\begin{split}
\lefteqn{\omega \left( X^{m}\otimes X^{r}\otimes X^{s}\right) \omega \left(X^{n}\otimes X^{m+r}\otimes X^{s}\right) \omega \left( X^{n}\otimes X^{m}\otimes X^{r}\right)} \\
&={\varphi \left( X^{r-1}\right) ^{\Sscript{-2}}\varphi \left( X^{m+r-1}\right) \varphi \left( X^{s+r-1}\right) \varphi \left( X^{m+r-1}\right) ^{\Sscript{-2}}\varphi \left( X^{n+m+r-1}\right) \varphi \left( X^{s+m+r-1}\right)\varphi \left(X^{m-1}\right) ^{\Sscript{-2}} \varphi \left( X^{n+m-1}\right) \varphi \left(X^{r+m-1}\right)}{} \\
&=\varphi \left( X^{m-1}\right) ^{\Sscript{-2}}\varphi \left( X^{n+m-1}\right)
\varphi \left( X^{s+m+r-1}\right) \varphi \left( X^{r-1}\right) ^{\Sscript{-2}}\varphi
\left( X^{s+r-1}\right) \varphi \left( X^{n+m+r-1}\right) \\
&=\omega \left( X^{n}\otimes X^{m}\otimes X^{r+s}\right) \omega \left(
X^{n+m}\otimes X^{r}\otimes X^{s}\right) .
\end{split}
\end{equation*}
This proves that $\omega$ is a $3$-cocycle. If $\omega\in \Bbbk[X]^{\bullet }\otimes \Bbbk [X]^{\bullet }\otimes \Bbbk
[X]^{\bullet }$, then
$$\varphi=\omega(-\otimes X\otimes X)=\left(\Bbbk[X]^\bullet \otimes \eta(X)\otimes \eta(X)\right)(\omega)\in \Bbbk[X]^\bullet$$
where $\eta=\eta_{\Bbbk[X]}$ is the map defined in equation \eqref{eq:unitdef}, a contradiction.
Since the comultiplication $\Delta $ is cocommutative, the datum $\left(
\Bbbk [X],m,u,\Delta ,\varepsilon ,\omega \right) $ defines a dual
quasi-bialgebra whose reassociator does not split, as desired. An example of a map $\varphi $ as above is exhibited in Lemma \ref{lemma:fact}.
\end{example}

\begin{remark}\label{remark:final}
As  we mentioned above, starting from a quasi-bialgebra  $(U,\Delta,\varepsilon,m,u,\omega)$, the construction of a reassociator for $U^{\bullet}$ is not at all clear and perhaps an impossible task. This in fact is connected   to a certain problem of localization in non-commutative algebras as follows. Precisely, we are asking for the construction of an invertible element $\Phi$ in a certain algebra $R$ (in our case $R=(U^{\bullet})^{\Sscript{\otimes 3}}$), by only knowing the existence of an invertible element $\omega$ in an algebra extension $T$ of $R$ (in our case $T=(U^{\Sscript{\otimes 3}})^*$ using the algebra map $\zeta: R \to T$ of equation \eqref{Eq:Gamma}).  In our opinion this construction is not at all realistic except perhaps in some very concrete situation.
This is why we think that Theorem \ref{them:main} was not established in a naive way and that it is the best result which can be extracted from this theory.
\end{remark}

\appendix

\section{New characterization}

In this section we give an alternative description of the finite dual in the non-associative case. Given a linear map, several useful criteria are shown in order to guarantee that this map belongs to the finite dual. Further characterizations can be found in \cite{Anquela}.

Given a vector space $V$ and  $S \subseteq V^*$, we denote by
$$S^{\Sscript{\perp}}\coloneqq \big\{ v \in V |\,\, s(v)=0, \forall s \in S \big\}.$$

For every $a\in A$ in an algebra $A$ and $f\in A^{\ast },$ we define in $A^{\ast }$ the elements $a\rightharpoonup f$ and $f\leftharpoonup a$ by setting, for every $b\in A$
\begin{equation}\label{Eq:quantumComp}
\left( a\rightharpoonup f\right) \left( b\right) \coloneqq f\left( ba\right) \qquad \text{and}\qquad \left( f\leftharpoonup a\right) \left( b\right) \coloneqq f\left( ab\right).
\end{equation}
Furthermore, the vector subspace of $A^*$ generated by the set $ \{ a\rightharpoonup f|\, a \in A\}$ will be simply denoted by $A\rightharpoonup f$. A similar notation will be adopted for the right action $\leftharpoonup$. The subsequent lemma is an analogue of \cite[Proposition 6.0.3]{Sweedler} or \cite[Lemma 9.1.1]{Montgomery}  and can be proved by the same argument.

\begin{lemma}\label{lem:A->f}
Let $f\in A^{\ast }.$ Then  the following are equivalent.
\begin{itemize}
\item[$(1)$] $m^{\ast }\left( f\right) \in \mathrm{Im}\left( \varphi
_{A,\, A}\right) $.

\item[$(2)$] $\dim _{\Bbbk }\big( A\rightharpoonup f\big) <\infty $.

\item[$\left( 3\right) $] $\dim _{\Bbbk }\big( f\leftharpoonup A\big)
<\infty $.
\end{itemize}
\end{lemma}

One cannot expect, as in the case of associative algebras \cite[Lemma 9.1.1]{Montgomery}, that the equivalent conditions (1)-(3) in Lemma \ref{lem:A->f}, imply either that $\dim _{\Bbbk }\Big( A\rightharpoonup ( f\leftharpoonup A)\Big)<\infty $  or that $\dim _{\Bbbk }\Big( (A\rightharpoonup  f)\leftharpoonup A\Big)<\infty $.  Nevertheless, the converse remains true.

\subsection{The tensor algebra and finite codimensional subspaces}\label{not:blacktriangle}
Let $V$ and $W$ be vector spaces endowed with a
$\Bbbk $-linear map $\phi _{V,W}^{1}:V\rightarrow \mathrm{End}_{\Bbbk }\left(
W\right) .$ Then this map induces a unique algebra map $\phi _{V,W}:T\left(V\right) \rightarrow \mathrm{End}_{\Bbbk }\left( W\right) ^{\mathrm{op}}$
such that $\left( \phi _{V,W}\right) _{\mid V}=\phi _{V,W}^{1}\,$and $\left(\phi _{V,W}\right) _{\mid \Bbbk }$ is the unit $\Bbbk \rightarrow \mathrm{End}_{\Bbbk }\left( W\right) ^{\mathrm{op}}:k\mapsto k\id{W},$ where $T(-)$ stands for the tensor algebra functor.

Then $W$ becomes a right $T\left(
V\right) $-module via $\blacktriangleleft $ defined, for every $z\in T\left(
V\right) ,w\in W$, by setting
\begin{equation*}
w\blacktriangleleft z\coloneqq \phi _{V,W}\left( z\right) \left( w\right) .
\end{equation*}%
Hence we can consider the left $T\left( V\right) $-module structure on $W^{\ast }$ uniquely defined by setting
\begin{equation*}
\left( z\blacktriangleright f\right) \left( w\right) \coloneqq f\left(
w\blacktriangleleft z\right) ,\text{ for every }z\in T\left( V\right) ,w\in
W,f\in W^{\ast }.
\end{equation*}

\begin{example}
Consider the so-called enveloping algebra $A^{\mathrm{e}}\coloneqq A\otimes A^{\mathrm{op}}$ as $V$ and $A$ as $W$. Then one can consider the map
\begin{equation*}
\phi _{V,W}^{1}:A^{\mathrm{e}}\rightarrow \mathrm{End}_{\Bbbk }\left(
A\right) :l\otimes r\mapsto \left[ a\mapsto r\left( al\right) \right] .
\end{equation*}
For shortness, we set
\begin{equation*}
\phi _{A}^{1}\coloneqq \phi _{V,W}^{1}\qquad \text{and}\qquad \phi _{A}\coloneqq \phi _{V,W}.
\end{equation*}
In particular, for every $l,r\in A,$ we get
\begin{equation}
x\blacktriangleleft \left( l\otimes r\right) =\phi _{A}\left( l\otimes
r\right) \left( x\right) =\phi _{A}^{1}\left( l\otimes r\right) \left(
x\right) =r\left( xl\right)  \label{form:Black2}
\end{equation}
and
\begin{equation*}
\left( \left( l\otimes r\right) \blacktriangleright f\right) \left( a\right)
=f\left( a\blacktriangleleft \left( l\otimes r\right) \right) \overset{(\ref{form:Black2})}{=}f\left( r\left( al\right) \right) =\left( l\rightharpoonup
\left( f\leftharpoonup r\right) \right) \left( a\right)
\end{equation*}%
so that%
\begin{equation}
\left( \left( l\otimes r\right) \blacktriangleright f\right) =\left(
l\rightharpoonup \left( f\leftharpoonup r\right) \right). \label{form:Balck4}
\end{equation}
\end{example}

For a subset $S \subseteq T(A^{\rm e})$ and an element $f \in A^*$, we denote by $ S  \blacktriangleright f$ the vector subspace of $A^*$ spanned by the set of elements $\{ s \blacktriangleright f|\, s \in S\}$.

\begin{proposition}\label{pro:dim}
Let $\left( A,m,u\right) $ be in $\nalgk$ . Then
\begin{equation*}
A^{\bullet }\,\,\overset{\eqref{MicFinDual}}{=}\,\, \sum_{V\in \mathcal{G}}V \,\,=\,\, \Big\{ f\in A^{\ast }\mid \dim _{\Bbbk }\left( \left( A^{\mathrm{e}}\right) ^{\otimes n}\blacktriangleright f\right) <\infty,\, \text{ for every }\, n\in\, \mathbb{N}\Big\} .
\end{equation*}
\end{proposition}

\begin{proof}
Set $T\coloneqq T\left( A^{\mathrm{e}}\right) $. We write a generator of $\left( A^{\mathrm{e}}\right) ^{\otimes i}$ in the
form $\left( l_{1}\otimes r_{1}\right) \otimes \cdots \otimes \left(
l_{i}\otimes r_{i}\right) $ where $l_{1},\ldots ,l_{i}\in A$ and $%
r_{1},\ldots ,r_{i}\in A^{\mathrm{op}}.$ Note that%
\begin{eqnarray*}
\left[ \phi _{A}\left( 1\otimes r\right) \circ \phi _{A}\left( l\otimes
1\right) \right] \left( a\right) &=&\left[ \phi _{A}^{1}\left( 1\otimes
r\right) \circ \phi _{A}^{1}\left( l\otimes 1\right) \right] \left( a\right)
\\
&=&\phi _{A}^{1}\left( 1\otimes r\right) \left( al\right) =r\left( al\right)
\\
&=&\phi _{A}^{1}\left( l\otimes r\right) \left( a\right) =\phi _{A}\left(
l\otimes r\right) \left( a\right)
\end{eqnarray*}
and hence
\begin{equation*}
\phi _{A}\left( l\otimes r\right) =\phi _{A}\left( 1\otimes r\right) \circ
\phi _{A}\left( l\otimes 1\right) =\phi _{A}\left( l\otimes 1\right) \circop\phi _{A}\left( 1\otimes r\right),
\end{equation*}
where the notation $\circop$ stands for  the multiplication of $\mathrm{End}_{\Bbbk }\left(A\right)^{ \rm op}$. Thus
\begin{eqnarray*}
\phi _{A}\left[ \left( l_{1}\otimes r_{1}\right) \otimes \cdots \otimes
\left( l_{i}\otimes r_{i}\right) \right] & = &\phi _{A}\left[ \left(
l_{1}\otimes r_{1}\right) \cdot _{T}\cdots \cdot _{T}\left( l_{i}\otimes
r_{i}\right) \right] \\
&=&\phi _{A}\left( l_{1}\otimes r_{1}\right) \circop\cdots
\circop\phi _{A}\left( l_{i}\otimes r_{i}\right) \\
&=&\phi _{A}\left( l_{1}\otimes 1\right) \circop\phi _{A}\left(
1\otimes r_{1}\right) \circop\cdots \circop\phi
_{A}\left( l_{i}\otimes 1\right) \circop\phi _{A}\left(
1\otimes r_{i}\right) \\
&=&\phi _{A}\left[ \left( l_{1}\otimes 1\right) \cdot _{T}\left( 1\otimes
r_{1}\right) \cdot _{T}\cdots \cdot _{T}\left( l_{i}\otimes 1\right) \cdot
_{T}\left( 1\otimes r_{i}\right) \right] \\
&=&\phi _{A}\left[ \left( l_{1}\otimes 1\right) \otimes \left( 1\otimes
r_{1}\right) \cdots \otimes \left( l_{i}\otimes 1\right) \otimes \left(
1\otimes r_{i}\right) \right]
\end{eqnarray*}
where the notation $._{T}$ stands for  the multiplication of  $T$. Therefore
\begin{eqnarray*}
a\blacktriangleleft \left[ \left( l_{1}\otimes r_{1}\right) \otimes \cdots
\otimes \left( l_{i}\otimes r_{i}\right) \right]
&=&\phi _{A}\left[ \left( l_{1}\otimes r_{1}\right) \otimes \cdots \otimes
\left( l_{i}\otimes r_{i}\right) \right] \left( a\right) \\
&=&\phi _{A}\left[ \left( l_{1}\otimes 1\right) \otimes \left( 1\otimes
r_{1}\right) \cdots \otimes \left( l_{i}\otimes 1\right) \otimes \left(
1\otimes r_{i}\right) \right] \left( a\right) \\
&=&a\blacktriangleleft \left[ \left( l_{1}\otimes 1\right) \cdot _{T}\left(
1\otimes r_{1}\right) \cdot _{T}\cdots \cdot _{T}\left( l_{i}\otimes
1\right) \cdot _{T}\left( 1\otimes r_{i}\right) \right]
\end{eqnarray*}%
Set $L\coloneqq A\otimes 1$ and $R\coloneqq1\otimes A^{\mathrm{op}}$. For shortness we
write $l\in L$ for $l\otimes 1$ and $r\in R$ for $1\otimes r.$ We also omit
the product over $T.$ Using this notation, we obtain%
\begin{equation*}
a\blacktriangleleft \left[ \left( l_{1}\otimes r_{1}\right) \otimes \cdots
\otimes \left( l_{i}\otimes r_{i}\right) \right] =a\blacktriangleleft \left(
l_{1}r_{1}\cdots l_{i}r_{i}\right) .
\end{equation*}%
For every $n\geq 1,f\in A^{\ast }$, we set
\begin{equation*}
W_{n}\left( f\right) \coloneqq \mathrm{Span}_{\Bbbk }\Big\{ \left(
a_{1}a_{2}\cdots a_{n-1}a_{n}\right) \blacktriangleright f\mid
a_{1},\ldots ,a_{n}\in L\cup R\Big\} .
\end{equation*}%
Set also $W_{0}\left( f\right) \coloneqq \Bbbk f$. Since both $A$ and $A^{\mathrm{op}%
}$ contain $1,$ it is clear that $W_{i}\left( f\right) \subseteq W_{j}\left(
f\right) $ for $i\leq j$.

Note further that $W_{i}\left( f\right) \subseteq \left( \left( A^{\mathrm{e}%
}\right) ^{\otimes i}\blacktriangleright f\right) \subseteq W_{2i}\left( f\right)$ for
every $i\in \mathbb{N}$ so that $\dim _{\Bbbk }\left( \left( A^{\mathrm{e}%
}\right) ^{\otimes n}\blacktriangleright f\right) <\infty $ if and only if $%
\dim _{\Bbbk }\left( W_{n}\left( f\right) \right) <\infty $ for every $n\in
\mathbb{N}$. Set
\begin{equation*}
B\coloneqq \Big\{ f\in A^{\ast }\mid \dim _{\Bbbk }\left( W_{n}\left( f\right)
\right) <\infty \text{ for every }\,n\in\, \mathbb{N}\Big\} .
\end{equation*}%
It remains to prove that $A^{\bullet }=B.$

$\subseteq )$ It suffices to prove that $V\subseteq B$ for every $V\in
\mathcal{G}$. Let us prove that $W_{n}\left( f\right) $ is finite
dimensional for every $f\in V$ by induction on $n\in \mathbb{N}$. For $n=0$
there is nothing to prove.

Let $n>0$ be such that $W_{n-1}\left( v\right) $ is finite-dimensional for
every $v\in V$. Let $f\in V.$ Write $\Delta _{V}\left( f\right)
=\sum_{i=1}^{t}g_{i}\otimes h_{i}\in V\otimes V.$ Let $a_{1},\ldots
,a_{n}\in L\cup R\ $and $w\coloneqq a_{1}a_{2}\cdots a_{n-1}.$ Then
\begin{equation*}
\left( \left( wa_{n}\right) \blacktriangleright f\right) \left( x\right)
=f\left( x\blacktriangleleft \left( wa_{n}\right) \right) =f\left( \left(
x\blacktriangleleft w\right) \blacktriangleleft a_{n}\right)
\end{equation*}
If $a_{n}=l\in L,$ then
\begin{equation*}
\left( \left( wa_{n}\right) \blacktriangleright f\right) \left( x\right)
=f\left( \left( x\blacktriangleleft w\right) \blacktriangleleft l\right)
\overset{(\ref{form:Black2})}{=}f\left( \left( x\blacktriangleleft w\right)
l\right) =\sum_{i=1}^{t}g_{i}\left( x\blacktriangleleft w\right) h_{i}\left(
l\right) =\sum_{i=1}^{t}\left( w\blacktriangleright g_{i}\right) \left(
x\right) h_{i}\left( l\right)
\end{equation*}
so that $\left( wa_{n}\right) \blacktriangleright
f=\sum_{i=1}^{n}h_{i}\left( l\right) \cdot \left( w\blacktriangleright
g_{i}\right) \in \sum_{i=1}^{n}W_{n-1}\left( g_{i}\right) .$ If $a_{n}=r\in
R,$ then%
\begin{equation*}
\left( \left( wa_{n}\right) \blacktriangleright f\right) \left( x\right)
=f\left( \left( x\blacktriangleleft w\right) \blacktriangleleft r\right)
\overset{(\ref{form:Black2})}{=}f\left( r\left( x\blacktriangleleft w\right)
\right) =\sum_{i=1}^{t}g_{i}\left( r\right) h_{i}\left( x\blacktriangleleft
w\right) =\sum_{i=1}^{t}g_{i}\left( r\right) \left( w\blacktriangleright
h_{i}\right) \left( x\right)
\end{equation*}
so that
$$
\left( \left( wa_{n}\right) \blacktriangleright f\right)
=\sum_{i=1}^{t}g_{i}\left( r\right) \cdot \left( w\blacktriangleright
h_{i}\right) \in \sum_{i=1}^{t}W_{n-1}\left( h_{i}\right)
$$
Thus
$$
\left(a_{1}a_{2}a_{3}a_{4}\cdots a_{n-1}a_{n}\right) \blacktriangleright f\in \sum_{i=1}^{t}W_{n-1}\left( g_{i}\right) +\sum_{i=1}^{t}W_{n-1}\left(h_{i}\right)
$$ for every $a_{1},\ldots ,a_{n}\in L\cup R$, which means that
\begin{equation*}
W_{n}\left( f\right) \subseteq \sum_{i=1}^{t}W_{n-1}\left( g_{i}\right)
+\sum_{i=1}^{t}W_{n-1}\left( h_{i}\right) .
\end{equation*}
Since, by inductive hypothesis, the latter is finite-dimensional so is $%
W_{n}\left( f\right) $.

$\supseteq )$ Let $f\in B$ and let us prove that $V\coloneqq \left(
T\blacktriangleright f\right) $ is good (this implies $f=\left(
1\blacktriangleright f\right) \in V\subseteq A^{\bullet }$). Consider an element  $v\in V$.
Then there is $z\in T$ such that $v=z\blacktriangleright f$. Write $%
z\coloneqq \sum_{i=0}^{n}z_{i}$ with $z_{i}\in \left( A^{\mathrm{e}}\right)
^{\otimes i}$ so that
$$
v=z\blacktriangleright
f=\sum_{i=0}^{n}z_{i}\blacktriangleright f\in \sum_{i=0}^{n}\left( \left( A^{\mathrm{e}}\right) ^{\otimes i}\blacktriangleright f\right) \subseteq \left(\left( A^{\mathrm{e}}\right) ^{\otimes n}\blacktriangleright f\right).
$$
Henceforth  it is not restrictive to assume $z\in \left( A^{\mathrm{e}}\right)
^{\otimes n}.$ We have then that
\begin{equation*}
\left( A\rightharpoonup v\right) \overset{(\ref{form:Balck4})}{\subseteq }
\left( A^{\mathrm{e}}\right) \blacktriangleright v\subseteq \left( A^{\mathrm{e}}\right) \blacktriangleright \left( z\blacktriangleright f\right)
\subseteq \left( A^{\mathrm{e}}\right) \blacktriangleright \left( \left( A^{\mathrm{e}}\right) ^{\otimes n}\blacktriangleright f\right) \subseteq \left(
A^{\mathrm{e}}\right) ^{\otimes \left( n+1\right) }\blacktriangleright f
\end{equation*}
and the latter is finite-dimensional. Hence $\left( A\rightharpoonup
v\right) $ is finite-dimensional and,  by Lemma \ref{lem:A->f},  we have that $m^{\ast}\left( v\right) \in {Im}\left( \varphi _{A,A}\right) .$ Write $m^{\ast
}\left( v\right) =\sum_{i=1}^{n}g_{i}\otimes h_{i}\in A^{\ast }\otimes
A^{\ast }.$ By the proof of the same lemma, we can choose $g_{1},\ldots ,g_{n}$
to form a basis of $\left( A\rightharpoonup v\right) .$ Thus there exist $a_{1},\ldots ,a_{n}\in A$ such that $g_{i}\left( a_{j}\right) =\delta
_{i,j}$.
We compute
\begin{equation*}
\left( \left( 1\otimes a_{j}\right) \blacktriangleright v\right) \left(
x\right) \overset{(\ref{form:Balck4})}{=}\left( v\leftharpoonup a_{j}\right)
\left( x\right) =v\left( a_{j}x\right) =\sum_{i=1}^{n}g_{i}\left(
a_{j}\right) h_{i}\left( x\right) =h_{j}\left( x\right)
\end{equation*}%
so that $h_{j}=\left( 1\otimes a_{j}\right) \blacktriangleright v\in \left(
A^{\mathrm{e}}\blacktriangleright v\right) \subseteq V.$ We have so proved
that $m^{\ast }\left( v\right) =\sum_{i=1}^{n}g_{i}\otimes h_{i}\in A^{\ast
}\otimes V.$ A similar argument shows that $m^{\ast }\left( v\right) \in
V\otimes A^{\ast }$ and hence $m^{\ast }\left( v\right) \in \left( A^{\ast
}\otimes V\right) \cap \left( V\otimes A^{\ast }\right) =V\otimes V.$
\end{proof}

\begin{remark}\label{rem:Iass}
Let $f\in A^{\ast }$ be such that $f(I)=0$ for some finite
codimensional ideal in $A$ (an ideal in a non-associative
algebra  is just a $\Bbbk $-vector subspace such that $aI\subseteq I$ and $Ia\subseteq I$ for all $a\in A$). Let $l\otimes r\in A^{\mathrm{e}}$ and let
$x\in I$. We have that
\begin{equation*}
\left( \left(l\otimes r\right)\blacktriangleright f\right) (x)=f(r(xl))\subseteq f(I)=0.
\end{equation*}
Inductively, if $z\in \left( A^{\mathrm{e}}\right) ^{\otimes n}$, $z=\left(l_{1}\otimes r_{1}\right) \otimes \cdots \otimes \left( l_{n-1}\otimes r_{n-1}\right) \otimes \left( l_{n}\otimes r_{n}\right) =w\otimes \left(l_{n}\otimes r_{n}\right) $, then
\begin{equation*}
\left( z\blacktriangleright f\right) (x)=\left( \left(l_{n}\otimes
r_{n}\right)\blacktriangleright f\right) (x\blacktriangleleft w)=f(r_{n}\left(
\left( x\blacktriangleleft w\right) l_{n}\right) )\subseteq f(r_{n}\left(
Il_{n}\right) )\subseteq f(I)=0.
\end{equation*}
Therefore $\left( A^{\mathrm{e}}\right) ^{\otimes n}\blacktriangleright f$
is contained in $I^{\perp }$, that injects into $\left( \frac{A}{I}\right)
^{\ast }$, which has finite dimension for all $n\in \mathbb{N}$. Hence, if $f $ vanishes on a finite codimensional ideal of $A$, then $f\in A^{\bullet }$.  This is an alternative way to show that $A^\circ$ is contained in $A^\bullet$, see Remark \ref{rem:oclassic}.
\end{remark}

\begin{remark} Another description of $A^{\bullet}$ by using the so-called \emph{standard filtration} $\left( T_{\left( n\right)}\right) _{n\in \mathbb{N}}$ of $T\coloneqq T\left( A^{\mathrm{e}}\right)$, is also possible. Precisely, this filtration is defined  by setting $T_{\left( n\right) }\coloneqq \bigoplus
_{i=0}^{n}\left( A^{\mathrm{e}}\right) ^{\otimes i}$, where $\left( A^{\mathrm{e}}\right) ^{\otimes 0}\coloneqq \Bbbk $. Then
\begin{equation*}
A^{\bullet }=\Big\{ f\in A^{\ast }\mid \dim _{\Bbbk }\left( T_{\left(
n\right) }\blacktriangleright f\right) <\infty \text{ for every }n\in
\mathbb{N}\Big\} .
\end{equation*}
In fact
\begin{equation*}
\left( T_{\left( n\right) }\blacktriangleright f\right) \subseteq \left(
\left( \bigoplus _{i=0}^{n}\left( A^{\mathrm{e}}\right) ^{\otimes i}\right)
\blacktriangleright f\right) \subseteq \sum_{i=0}^{n}\left( \left( A^{\mathrm{e}}\right) ^{\otimes i}\blacktriangleright f\right) \subseteq \left(
\left( A^{\mathrm{e}}\right) ^{\otimes n}\blacktriangleright f\right)
\end{equation*}
so that $\left( T_{\left( n\right) }\blacktriangleright f\right) =\left(
\left( A^{\mathrm{e}}\right) ^{\otimes n}\blacktriangleright f\right) $.
\end{remark}

We now give a characterization of $A^{\bullet }$ in the spirit of \cite[Definition 1.2.3]{Montgomery}.

\begin{proposition}\label{prop:In}
Let $\left( A,m,u\right) $ be in $\nalgk$ and let $f\in A^{\ast }$. Then the following are equivalent
\begin{enumerate}[(i)]
\item $f\in A^{\bullet }$;
\item There is a family
$\left( I_{n}\right) _{n\in \mathbb{N}}$ of subspaces of $A$ of finite
codimension such that, for each $n\geq 1$,
$$\left( I_{n}\blacktriangleleft
A^{\mathrm{e}}\right) \subseteq I_{n-1},\; \text{ and }\;f\left( I_{0}\right) =0.$$
\end{enumerate}
Moreover, if one the these conditions holds true, then we can choose
$$
I_{0}=\mathrm{Ker}(f)\;\text{ and }\;I_{n}=\Big\{a\in A\mid a\blacktriangleleft A^{\mathrm{e}}\subseteq
I_{n-1}\Big\},\;\; \text{for every }\; n>0.
$$
\end{proposition}

\begin{proof}
$\left( \Rightarrow \right)$. Assume $f \in A^{\bullet}$ and set $I_{n}\coloneqq \left( \left( A^{\mathrm{e}
}\right) ^{\otimes n}\blacktriangleright f\right) ^{\perp }.$ For every $
n\geq 1,$ $u\in I_{n},z\in A^{\mathrm{e}},w\in \left( A^{\mathrm{e}}\right)
^{\otimes \left( n-1\right) },$
\begin{equation*}
\left( w\blacktriangleright f\right) \left( u\blacktriangleleft z\right)
=\left( z\blacktriangleright \left( w\blacktriangleright f\right) \right)
\left( u\right) =\left( \left( zw\right) \blacktriangleright f\right) \left(
u\right) \in \left( \left( A^{\mathrm{e}}\right) ^{\otimes
n}\blacktriangleright f\right) \left( u\right) =0
\end{equation*}
so that $u\blacktriangleleft z\in \left( \left( A^{\mathrm{e}}\right)
^{\otimes \left( n-1\right) }\blacktriangleright f\right) ^{\perp }=I_{n-1}$
and hence $\left( I_{n}\blacktriangleleft A^{\mathrm{e}}\right) \subseteq
I_{n-1}.$ Since $\left( \left( A^{\mathrm{e}}\right) ^{\otimes
0}\blacktriangleright f\right) =\left( \Bbbk \blacktriangleright f\right)
=\Bbbk f$ we get that $f\left( I_{0}\right) =0.$

$\left( \Leftarrow \right) .$ Inductively one proves that $\left(
I_{n}\blacktriangleleft \left( A^{\mathrm{e}}\right) ^{\otimes n}\right)
\subseteq I_{0}$ so that we have
\begin{equation*}
\left( \left( A^{\mathrm{e}}\right) ^{\otimes n}\blacktriangleright f\right)
\left( I_{n}\right) \subseteq f\left( I_{n}\blacktriangleleft \left( A^{
\mathrm{e}}\right) ^{\otimes n}\right) \subseteq f\left( I_{0}\right) =0.
\end{equation*}
Therefore $\left( \left( A^{\mathrm{e}}\right) ^{\otimes
n}\blacktriangleright f\right) \subseteq I_{n}^{^{\perp }}$ which is
finite-dimensional as $I_{n}$ has finite codimension, which by Proposition \ref{pro:dim} implies that $f \in A^{\bullet}$.

Let us check the last statement. For $n=0$ we have that
\begin{equation*}
I_{0}=\left( \left( A^{\mathrm{e}}\right) ^{\otimes 0}\blacktriangleright
f\right) ^{\perp }=\left( \Bbbk \blacktriangleright f\right) ^{\perp
}=\left( \Bbbk f\right) ^{\perp }=\mathrm{Ker}(f),
\end{equation*}%
and for $n>0$ we have:%
\begin{eqnarray*}
I_{n} &\coloneqq &\left( \left( A^{\mathrm{e}}\right) ^{\otimes
n}\blacktriangleright f\right) ^{\perp }=\Big\{a\in A\mid \left(\left( A^{\mathrm{e}
}\right) ^{\otimes n}\blacktriangleright f\right)(a)=0\Big\} \\
&=&\Big\{a\in A\mid \left(\left( A^{\mathrm{e}}\right) ^{\otimes
(n-1)}\blacktriangleright f\right)(a\blacktriangleleft A^{\mathrm{e}})=0\Big\} \\
&=&\Big\{a\in A\mid (a\blacktriangleleft A^{\mathrm{e}})\subseteq \left( \left(
A^{\mathrm{e}}\right) ^{\otimes (n-1)}\blacktriangleright f\right) ^{\perp
}\Big\} \\
&=&\Big\{a\in A\mid (a\blacktriangleleft A^{\mathrm{e}})\subseteq I_{n-1}\Big\},
\end{eqnarray*}
and this finishes the proof.
\end{proof}

Let $C$ be a coalgebra. Then the coalgebra structure of $C$, through the universal property of the tensor algebra, induces a bialgebra structure on $T\left( C\right) $ so that it makes sense to use the notation $\Delta_{\Sscript{T(C)}}(z)\coloneqq \sum
z_{1}\otimes z_{2}$ for any $z\in T\left( C\right)$, for the comultiplication of  $T\left( C\right) $; see e.g. \cite[Theorem 5.3.1]{Radford-Hopf}.

\begin{lemma}
\label{lem:modalg}Let $C$ and $D$ be two coalgebras with a $\Bbbk $-linear map $\phi _{C,\,D}^{1}:C\rightarrow \mathrm{End}_{\Bbbk }\left(D\right) $ as in subsection \ref{not:blacktriangle}. Assume that $D\otimes C\rightarrow D:d\otimes c\mapsto d\blacktriangleleft c$ is a coalgebra map. Then $D$ is a right $T\left( C\right) $-module coalgebra through $\blacktriangleleft $ and $\left( D^{\ast },m_{D^{\ast }},u_{D^{\ast
}}\right) $ is a left $T\left( C\right) $-module algebra through $
\blacktriangleright $ where
\begin{equation*}
m_{D^{\ast }}\left( f\otimes g\right) =f\ast g\text{ (convolution product),}
\; \text{and}\; u_{D^{\ast }}\left( k\right) =k\varepsilon _{D}
\end{equation*}
for every $f,g\in D^{\ast },k\in \Bbbk $.
\end{lemma}

\begin{proof}
By hypothesis  for every $c\in C,d\in D,$ we have that
\begin{eqnarray*}
\left( d_{1}\blacktriangleleft c_{1}\right) \otimes \left(
d_{2}\blacktriangleleft c_{2}\right) &=&\left( d\blacktriangleleft c\right)_{1}\otimes \left( d\blacktriangleleft c\right) _{2}, \\
\varepsilon _{D}\left( d\right) \varepsilon _{C}\left( c\right)
&=&\varepsilon _{D}\left( d\blacktriangleleft c\right) .
\end{eqnarray*}
We need to prove that for every $z\in T\left( C\right) ,d\in D,$ we have
\begin{eqnarray*}
\left( d_{1}\blacktriangleleft z_{1}\right) \otimes \left(
d_{2}\blacktriangleleft z_{2}\right) &=&\left( d\blacktriangleleft z\right)_{1}\otimes \left( d\blacktriangleleft z\right) _{2}, \\
\varepsilon _{D}\left( d\right) \varepsilon _{T\left( C\right) }\left(
z\right) &=&\varepsilon _{D}\left( d\blacktriangleleft z\right) .
\end{eqnarray*}

For $k\in \Bbbk $ we have
\begin{equation*}
d\blacktriangleleft k=\phi _{C,D}\left( k\right) \left( d\right) =\left( k
\id{D}\right) \left( d\right) =kd.
\end{equation*}
Then, for every $z\in \Bbbk $ we have
\begin{eqnarray*}
\left( d_{1}\blacktriangleleft z_{1}\right) \otimes \left(
d_{2}\blacktriangleleft z_{2}\right) &=&\left( d_{1}\blacktriangleleft
z1\right) \otimes \left( d_{2}\blacktriangleleft 1\right) =zd_{1}\otimes
d_{2} \\
&=&\left( zd\right) _{1}\otimes \left( zd\right) _{2}=\left(
d\blacktriangleleft z\right) _{1}\otimes \left( d\blacktriangleleft z\right)_{2},
\end{eqnarray*}
and
\begin{equation*}
\varepsilon _{D}\left( d\right) \varepsilon _{T\left( C\right) }\left(
z\right) =\varepsilon _{D}\left( d\right) z=\varepsilon _{D}\left( dz\right)=\varepsilon _{D}\left( d\blacktriangleleft z\right) .
\end{equation*}
Let $c_{1},\ldots ,c_{n}\in C.$ Let us prove, by induction on $n\geq 1$,  that
$$
\left(d_{1}\blacktriangleleft z_{1}\right) \otimes \left( d_{2}\blacktriangleleft
z_{2}\right) =\left( d\blacktriangleleft z\right) _{1}\otimes \left(
d\blacktriangleleft z\right) _{2}\; \text{ and } \;
\varepsilon _{D}\left( d\right)
\varepsilon _{T\left( C\right) }\left( z\right) =\varepsilon _{D}\left(
d\blacktriangleleft z\right), $$
where $z\coloneqq c_{1}\cdots c_{n}$  is the multiplication of the $c_i$'s, each  one viewed as an element in $T(C)$.

For $n=1$ there is nothing to prove. Let $n>1$ and assume the statement true for $n-1$. If we set $z^{\prime }\coloneqq c_{1}\cdots c_{n-1}$, then
we get  form one hand that
\begin{eqnarray*}
\left( d_{1}\blacktriangleleft z_{1}\right) \otimes \left(
d_{2}\blacktriangleleft z_{2}\right) &=&\left( d_{1}\blacktriangleleft
\left( z^{\prime }c_{n}\right) _{1}\right) \otimes \left(
d_{2}\blacktriangleleft \left( z^{\prime }c_{n}\right) _{2}\right) =\left(
d_{1}\blacktriangleleft z_{1}^{\prime }\left( c_{n}\right) _{1}\right)
\otimes \left( d_{2}\blacktriangleleft z_{2}^{\prime }\left( c_{n}\right)
_{2}\right) \\
&=&\left( \left( d_{1}\blacktriangleleft z_{1}^{\prime }\right)
\blacktriangleleft \left( c_{n}\right) _{1}\right) \otimes \left( \left(
d_{2}\blacktriangleleft z_{2}^{\prime }\right) \blacktriangleleft \left(
c_{n}\right) _{2}\right) \\
&=&\left( \left( d\blacktriangleleft z^{\prime }\right)
_{1}\blacktriangleleft \left( c_{n}\right) _{1}\right) \otimes \left( \left(
d\blacktriangleleft z^{\prime }\right) _{2}\blacktriangleleft \left(
c_{n}\right) _{2}\right) \\
&=&\left( \left( d\blacktriangleleft z^{\prime }\right) \blacktriangleleft
c_{n}\right) _{1}\otimes \left( \left( d\blacktriangleleft z^{\prime
}\right) \blacktriangleleft c_{n}\right) _{2} \\
&=&\left( d\blacktriangleleft \left( z^{\prime }c_{n}\right) \right)
_{1}\otimes \left( d\blacktriangleleft \left( z^{\prime }c_{n}\right)
\right) _{2}=\left( d\blacktriangleleft z\right) _{1}\otimes \left(
d\blacktriangleleft z\right) _{2}
\end{eqnarray*}
and from the other hand that
\begin{eqnarray*}
\varepsilon _{D}\left( d\right) \varepsilon _{T\left( C\right) }\left(
z\right) &=&\varepsilon _{D}\left( d\right) \varepsilon _{T\left( C\right)
}\left( z^{\prime }c_{n}\right) =\varepsilon _{D}\left( d\right) \varepsilon_{T\left( C\right) }\left( z^{\prime }\right) \varepsilon _{T\left( C\right)
}\left( c_{n}\right) \\
&=&\varepsilon _{D}\left( d\blacktriangleleft z^{\prime }\right) \varepsilon_{T\left( C\right) }\left( c_{n}\right) =\varepsilon _{D}\left( \left(d\blacktriangleleft z^{\prime }\right) \blacktriangleleft c_{n}\right) \\
&=&\varepsilon _{D}\left( d\blacktriangleleft \left( z^{\prime }c_{n}\right) \right) =\varepsilon _{D}\left( d\blacktriangleleft z\right) .
\end{eqnarray*}
This shows the claimed formulae for every $z\in T\left( C\right) $.
Therefore $D$ is a right $T\left( C\right) $-module coalgebra through $\blacktriangleright $. Since $\left( D,\Delta _{D},\varepsilon _{D}\right) $ is a coassociative coalgebra, we know that $\left( D^{\ast },m_{D^{\ast
}},u_{D^{\ast }}\right) $ is an associative algebra. Let us check that it is
a left $T\left( C\right) $-module algebra through $\blacktriangleleft$.
For all $f,g\in D^{\ast },z\in T\left( C\right) ,d\in D$ we have%
\begin{eqnarray*}
\sum \Big[ \left( z_{1}\blacktriangleright f\right) \ast \left(
z_{2}\blacktriangleright g\right) \Big] \left( d\right) &=&\sum \left(
z_{1}\blacktriangleright f\right) \left( d_{1}\right) \left(
z_{2}\blacktriangleright g\right) \left( d_{2}\right) \\
&=&\sum f\left( d_{1}\blacktriangleleft z_{1}\right) g\left(
d_{2}\blacktriangleleft z_{2}\right) =\sum f\Big( \left(
d\blacktriangleleft z\right) _{1}\Big) g\Big( \left( d\blacktriangleleft
z\right) _{2}\Big) \\
&=&\left( f\ast g\right) \left( d\blacktriangleleft z\right) =\left(
z\blacktriangleright \left( f\ast g\right) \right) \left( d\right)
\end{eqnarray*}
so that $\sum \left( z_{1}\blacktriangleright f\right) \ast \left(
z_{2}\blacktriangleright g\right) =\left( z\blacktriangleright \left( f\ast
g\right) \right) .$ Moreover
\begin{equation*}
\left( z\blacktriangleright \varepsilon _{D}\right) \left( d\right)
=\varepsilon _{D}\left( d\blacktriangleleft z\right) =\varepsilon _{D}\left(d\right) \varepsilon _{T\left( C\right) }\left( z\right)
\end{equation*}
so that $z\blacktriangleright \varepsilon _{D}=\varepsilon _{T\left(
C\right) }\left( z\right) \varepsilon _{D}.$ This proves that $\left(
D^{\ast },m_{D^{\ast }},u_{D^{\ast }}\right) $ is a left $T\left( C\right) $-module algebra through $\blacktriangleright $.
\end{proof}

\begin{remark}\label{remark:mfrak}
More generally, given a bialgebra $B,$ the obvious contravariant functor $\left( -\right) ^{\ast }:\mathfrak{M}{_{B}\rightarrow {_{B}}}\mathfrak{M}$, from the category of right to the category of left $B$-modules, is lax monoidal so that it induces the covariant functor $\left( -\right)^{\ast }:\mathfrak{M}{_{B}}\rightarrow \left( {{_{B}}}\mathfrak{M}\right) ^{\mathrm{op}}$ which is colax monoidal. Thus the latter functor induces a functor \textrm{Coalg}$\left( \left( -\right) ^{\ast }\right) :\mathrm{Coalg}
\left( \mathfrak{M}{_{B}}\right) \rightarrow \mathrm{Coalg}\left( \left( {{_{B}}}\mathfrak{M}\right) ^{\mathrm{op}}\right) \equiv \left( \mathrm{Alg} \left( {{_{B}}}\mathfrak{M}\right) \right) ^{\mathrm{op}}$ which means that $\left( -\right) ^{\ast }$ maps right $B$-module coalgebras to left $B$-module algebras as in the particular case of Lemma \ref{lem:modalg}.
\end{remark}

\section{Complementary results}

The following result is probably well-known but we were not able to find a reference.

\begin{lemma}\label{lemma:unitgraded}
Let $A=\bigoplus_{n\in \mathbb{N}}A_n$ be an $\mathbb{N}$-graded ring. Suppose that the product of two non-zero homogeneous elements is non-zero. Then the invertible element of $A$ are concentrated in $A_0$. Moreover, $A$ is a domain.
\end{lemma}

\begin{proof}
Let $x,y\in A$ be non-zero elements. Write $x=x_0+x_1+\cdots+x_s$, where $x_i\in A_i$, and $y=y_0+y_1+\cdots+y_t$, where $y_i\in A_i$, with $x_s\neq 0$ and $y_t\neq 0$. By assumption, $x_sy_t\neq 0$ and it is clearly the homogeneous element with greatest degree of $xy$.

If $xy=1$, then the only possibility is $s+t=0$ whence $s=0$ which means $x\in A_0$.

If $xy=0$, then we must have $x_sy_t=0$, which is a contradiction.
\end{proof}

\begin{corollary}\label{corollario:tdominio}
Given a vector space $V$, the group of units of the tensor algebra $T(V)$ is $\Bbbk\setminus\{0\}$ and $T(V)$ is a domain.
\end{corollary}

\begin{proof}
We have that $T=T(V)$ is graded with respect to $T_n\coloneqq V^{\otimes n}$. Given $x\in T_s$ and $y\in T_t$ non-zero elements, we have that $x\cdot y=x\otimes y$ which is non-zero.
\end{proof}

\begin{lemma}\label{lemma:dominio}
Let $R$ be a $\Bbbk$-algebra that is also a domain. Then $T(V)\otimes R$ is a domain.
\end{lemma}

\begin{proof}
Set $T=T(V)$. By the Axiom of Choice we can choose a totally ordered basis $\mathcal{B}_V\coloneqq\{v_i\mid i\in I\}$ for $V$. Mimiking \cite[Example 2 and 3]{Green} we can construct an admissible graded lexicographic order on the basis $\mathcal{B}_T\coloneqq\{v_{i_1}v_{i_2}\cdots v_{i_n}\mid n\geq 1\textrm{ and }i_1,\ldots,i_n\in I\}\cup\{1_{\Bbbk}\}$ of $T(V)$ as follows
\begin{equation*}
v_{i_1}v_{i_2}\cdots v_{i_n} < v_{j_1}v_{j_2}\cdots v_{j_m}
\end{equation*}
if $n<m$ or $n=m$, $v_{i_s}=v_{j_s}$ for $0\leq s\leq (t-1)<n$ and $v_{i_{t}}<v_{j_{t}}$ with respect to the total order on $\mathcal{B}$. Let $x,y\in T\otimes R$ with $x\neq 0$ and $y\neq 0$. We can write $x=b_{i_1}\otimes x_{1}+\cdots +b_{i_s}\otimes x_{s} $ where $x_1,\ldots,x_s\in R$ with $x_s\neq 0$, $b_{i_1},\cdots, b_{i_s}\in \mathcal{B}_T$ and $b_{i_1}<\cdots < b_{i_s}$. Analogously write $y=b_{j_1}\otimes y_{1}+\cdots +b_{j_t}\otimes y_{t} $ where $y_1,\ldots,y_t\in R$ with $y_t\neq 0$, $b_{j_1},\cdots, b_{j_t}\in \mathcal{B}_T$ and $b_{j_1}<\cdots < b_{j_t}$. Since $R$ is a domain, $x_sy_t\neq 0$. Moreover, $b_{i_s}b_{j_t}\in \mathcal{B}_T$ whence $\left(b_{i_s}\otimes x_{s}\right)\left(b_{j_t}\otimes y_{t}\right)=b_{i_s}b_{j_t}\otimes x_{s}y_{t}\neq 0$. Note that $xy=b_{i_1}b_{j_1}\otimes x_{1}y_1+\cdots +b_{i_s}b_{j_t}\otimes x_{s}y_{t}$ where $b_{i_s}b_{j_t}$ is the greatest of all the first entries of the summands involved. Thus $xy\neq 0$.
\end{proof}

\begin{corollary}\label{corollario:tensorn}
Given a vector space $V$ and $n\in \mathbb{N}$, the group of units of the tensor algebra $T(V)^{\otimes n}$ is $\Bbbk\setminus\{0\}$ and $T(V)^{\otimes n}$ is a domain.
\end{corollary}

\begin{proof}
By induction on $n$, in view of Lemma \ref{lemma:dominio}, $T(V)^{\otimes n}$ is a domain. Moreover, since $T(V)$ is a graded algebra, $T(V)^{\otimes n}$ is graded too. By Lemma \ref{lemma:unitgraded}, the group of units of $T(V)^{\otimes n}$ is concentrated in degree zero.
\end{proof}

\begin{remark}
Obviously Corollary \ref{corollario:tdominio} follows also by Corollary \ref{corollario:tensorn}.
\end{remark}

\begin{lemma}\label{lemma:fact}
Let $\Bbbk $ be a field and consider $\Bbbk [X]$ the (bi)algebra of
polynomials in one indeterminate $X$. The map $\varphi :\Bbbk \lbrack
X]\longrightarrow \Bbbk $ given by $\varphi \left( X^{n}\right) \coloneqq n!$ and
extended by linearity does not belong to $\Bbbk [X]^{\circ }$, the
ordinary finite dual of $\Bbbk [X]$.
\end{lemma}

\begin{proof}
Assume, by contradiction, that $\varphi \in \Bbbk [X]^{\circ }$. Then
there exists a (finite-codimensional) ideal $I\coloneqq \left\langle p\left(
X\right) \right\rangle$ in $\Bbbk [X]$ with $\varphi
\left( I\right) =0$. Consider the system of the equations $\varphi \left( X^{i}p\left( X\right) \right) =0$ for $i=0,\ldots,n$.
If we write $p\left( X\right) = \sum_{j=0}^{n}p_{j}X^{j}$, these equations become $\sum_{j=0}^{n}p_{j}\varphi\left(X^{i+j}\right)=0$ for $i=0,\ldots,n$.
The matrix associated to this system is
\begin{equation*}
T \coloneqq \left(
\begin{array}{ccccc}
0! & 1! & \cdots  & \left( n-1\right) ! & n! \\
1! & 2! & \cdots  & n! & \left( n+1\right) ! \\
\vdots  & \vdots  & \ddots  & \vdots  & \vdots  \\
\left( n-1\right) ! & n! & \cdots  & \left( 2n-2\right) ! & \left(
2n-1\right) ! \\
n! & \left( n+1\right) ! & \cdots  & \left( 2n-1\right) ! & \left( 2n\right)
!%
\end{array}%
\right)
\end{equation*}%
Thus $T=\Big( \left(i+j\right) !\Big)$ for $i,j$ that run
from $0$ to $n$. We claim that $\det \left( T\right) \neq 0$, or equivalently that $T$ is
invertible, which is impossible since $p(X)\neq 0$, as $I$ is finite-codimensional. To show this,
let us consider the $n$-th Pascal matrix $Q_{n}=\left( q_{ij}\right) $, i.e. the matrix whose entries are given by the relation $%
q_{ij}\coloneqq \binom{i+j}{i}$. Then:%
\begin{equation*}
\det \left( T\right)  = \det  \Big( \left(
i+j\right) !\Big)  \,=\,\det  \left( i!j!q_{ij}\right)
   =  \prod_{i=0}^{n}i!\prod_{j=0}^{n}j!\det \left(  Q_n
\right)
  =  \Big( 0!1!\cdots (n-1)!n!\Big) ^{2}\det
\left( Q_{n}\right).
\end{equation*}
In view of \cite[Discussion preceding
Theorem 4]{BrawerPivorino}, we know that $\det \left( Q_{n}\right) =1$,
whence
$\det(T)\neq0 $
and the claim is proved.
\end{proof}

\begin{remark}
The fact that the map $\varphi\left(X^n\right)=n!$ is not in $\Bbbk [X]^{\circ }$ seems to be well-known, see \cite[Section 2]{FutiaMullerTaft}. This depends on the correspondence between elements in $\Bbbk [X]^{\circ }$ and linearly recursive sequences, see e.g. \cite{LarsonTaft}. Since we could not find an explicit proof that $n!$ defines a non-linearly recursive sequence, we included the previous lemma.
\end{remark}

\smallskip

\textbf{Acknowledgements.}
L. El Kaoutit would like to thank all the members of the Department of Mathematics ``Giuseppe Peano'' for the warm hospitality and for the unsurpassable atmosphere  during his visit.

\end{document}